\documentclass[11pt]{amsart}

\usepackage{amsmath, amsthm, amssymb}
\usepackage[pdftex]{graphicx}
\usepackage{color}
\usepackage{enumerate}
\newtheorem{theorem}{Theorem}
\newtheorem{lemma}{Lemma}

\newtheorem{defn}{Definition}

\newtheorem{remark}{Remark}

\newcommand{\diver}{\text{div} \ }

\newcommand{\eqdef}{\overset{\mbox{\tiny{def}}}{=}}

\def\ss {\mu}
\def\Ddim {d}
\newcommand{\ba}{\begin{equation}}
\newcommand{\ea}{\end{equation}}
\newcommand{\bea}{\begin{eqnarray}}
\newcommand{\eea}{\end{eqnarray}}

\newcommand{\D}{\mathcal{D}}

\newcommand{\rr}{\mathcal{R}}

\newcommand{\Hper}{H }
\newcommand{\Lper}{L }
\newcommand{\Lty}{L^{\infty}}
\newcommand{\uv}{u_{\varepsilon}}
\newcommand{\uvt}{\partial_t  u_{\varepsilon}}
\newcommand{\uvk}{u_{\varepsilon_k}}

\newcommand{\Dld}{D_{1\delta}(t)}
\newcommand{\Dgd}{D_{2\delta}(t)}
\newcommand{\vark}{\varepsilon_k}

\newcommand{\ale}{a.e.\,\,}

\newcommand{\dpstyle}{\displaystyle}

\newcommand{\nn}{\nonumber}

\newcommand{\pd}{\partial}

\newcommand{\8}{\infty}

\newcommand{\ud}{\,\mathrm{d}}

 \newcommand{\pdf}{\pd_x^4}

  \newcommand{\I}{\mathbb{T}}

 \renewcommand{\d}{\,\text{d}}

\newcommand{\veps}{\varepsilon}

\definecolor{mygreen}{rgb}{0.1,0.75,0.2}

\title[Crystal Surface Models with Evaporation and Deposition]{Analysis of a continuum theory for broken bond crystal surface models with evaporation and deposition effects}

\author[Y. Gao]{Yuan Gao}
\address{Department of Mathematics, Duke University, Durham, NC}
\email{yuangao@math.duke.edu}

\author[J.-G. Liu]{Jian-Guo Liu}
\address{Department of Mathematics and Department of
  Physics, Duke University, Durham, NC}
\email{jliu@math.duke.edu}

\author[J. Lu]{Jianfeng Lu}
\address{Department of Mathematics, Department of Physics, and Department of Chemistry, Duke University, Durham, NC}
\email{jianfeng@math.duke.edu}

\author[J.L. Marzuola]{Jeremy L. Marzuola}
\address{Department of Mathematics, University of North Carolina at Chapel Hill \\ CB\#3250
  Phillips Hall \\ Chapel Hill, NC 27599}
\email{marzuola@math.unc.edu}

\begin{document}

\begin{abstract}
We study a $4$th order degenerate parabolic PDE model in $1$ dimension with a $2$nd order correction 
%that is a modification of one suggested by Peter Smereka using the tools of Krug et al (Z. Phys. B., 1994).  The PDE is meant to 
modeling the evolution of a crystal surface under the influence of both thermal fluctuations and evaporation/deposition effects.  First, we provide a non-rigorous derivation of the PDE from an atomistic model using variations on Kinetic Monte Carlo rates proposed by the last author with  Weare (PRE, 2013).  Then, we prove the existence of a global in time weak solution for the PDE by regularizing the equation in a way that allows us to apply the tools of Bernis-Friedman (JDE, 1990).  The methods developed here can be applied to a large number of $4$th order degenerate PDE models.    In an appendix, we also discuss the global smooth solution with small data in the Weiner algebra framework following recent developments using tools of the second author with Robert Strain (IFB, 2018). 
\end{abstract}

\maketitle

\begin{center}
  Dedicated to Peter Smereka (1959--2015) for his inspiration and
  insights.
\end{center}

\section{Introduction}

We explore here the limiting macroscopic evolution of a family of
microscopic models of dynamics on a $1$ dimensional crystal surface
experiencing fluctuations from both thermodynamic hopping as well as
evaporation/deposition effects.  The family of atomistic models from
which we start includes the well known (and well studied)
solid-on-solid (SOS) model \cite{Binh} and is remarkable, given its
simplicity, for its widespread use in large scale simulations of
crystal evolution \cite{pimpinelli1999physics}.  Our investigation
complements and extends the work by Krug, Dobbs, and Majaniemi in
\cite{KDM} and the last author with Weare \cite{MW1} on the
solid-on-solid model.  {  It is possible to extend the
  modeling and some of the PDE arguments to dimension $2$, but we work
  in the one-dimensional setting to make many of the calculations in
  the microscopic modeling section especially easier to follow
  notationally and also to allow clarity in the regularization
  arguments for the PDE analysis.}

We will study both the derivation of and the global solutions of the model with periodic boundary conditions given by \ba\label{pde}
\partial_t h = c_1\Delta e^{-\Delta h}+ c_2(1- e^{-\Delta h}), \quad
\mbox{in }\mathbb{T} \times (0, \infty) \ea with initial data
$h(x,0)=h_0$ of sufficient regularity to be discussed more carefully
below. Here $c_1>0$ and $c_2>0$ are physical constants which will be
set to $1$ for simplicity (after some suitable choice of units).

The derivation we present will follow very similarly the ideas of \textsc{Marzuola and Weare}
\cite{MW1} and \textsc{Smereka} \cite{Smereka}, while attempting to clarify some of the
choices of non-equilibrium dynamics and putting all concepts in the
notation characteristic of the statistical physics community. The
fourth order equation ($c_1 > 0, c_2 = 0$) is conservative and arises
from atoms hopping from one lattice site to the next with rates that
depend upon the local curvature.  The second order term stems from
interaction of the crystal with a gas of atoms and is a balance of the
effects of a constant rate of deposition and an evaporation rate that
is once again comparable to the local curvature.  The scalings of the
rates that make these two phenomena both comparable for large system
sizes will be discussed.

As noted above, the $4$th order component of the model arises where only thermodynamic fluctuations are considered.  Using a generalization of rates determined by bond breaking energies to describe a family of microscopic processes, a class of $4$th order PDEs with exponential mobility including \eqref{pde} with $c_2=0$ were derived and studied in \cite{MW1}.   The notion of mobility will be discussed in more detail in the analysis of \eqref{pde}, but essentially the mobility refers to the metric structure that arises an appropriately interpreted gradient descent approach to the dynamics.  In addition to being directly derived in \cite{MW1}, \eqref{pde} can be seen as a leading order approximation to the PDE model proposed in the last section of \cite{KDM} where the rates were build directly on a bond counting model.  Compared with the diffusion effect expressed by the $4$th order component of the model, the evaporation effect is a $2$nd order component first introduced in \cite{Spohn1993}.

Upon discussing the derivation of PDE model \eqref{pde}, we will
attempt to establish some properties of its solutions.  Much work has
recently gone into understanding PDEs of this type though mostly
without the $2$nd order correction.  Indeed, there has been recent
analytic progress in terms of global existence, characterization of
dynamics, construction of local solutions and classification of the
breakdown of regularity has been made on the related $H^{-1}$ steepest
descent flows predicted by the adatom rates with stemming from quadratic interaction potentials,
\begin{equation}
\label{Hm1expde}
\partial_t h = \Delta e^{- \Delta h}.
\end{equation}
See for instance
\cite{liu2016existence,liu2017analytical,gao2017gradient,xu2017existence,ambrose,LS2018,GBM2018},
each establishes existence in various ways and in the case of strong
solutions uniqueness. See also \cite{Gao2019, LLMM} for the case PDE models arising from rates that simply involve bond counting.  The methods applied involve a combination of
approaches to regularizing the model, Weiner algebra tools (for small
data highly regular solutions), modifying the tools from gradient
flows, etc.  We note that linearizing the exponential results in the
bi-Laplacian heat flow as the leading order flow, which could be a
means to prove local well-posedness using standard tools from
quasilinear parabolic equations.  However, there is a clear breakdown
of convex/concave symmetry for the model with exponential mobiility
that is not observed in the linear model for \eqref{Hm1expde} (see
\cite{MW1,LLMM}).

The paper will proceed as follows.  In Section \ref{deriv}, we describe the family of atomistic models that we consider and discuss the findings in \cite{KDM} in more detail.  Then, we proceed to follow the ideas laid out in \cite{MW1} to provide a framework for deriving \eqref{pde}.  In Section \ref{weak}, we prove global existence of weak solutions using a formulation of the modified biharmonic porous medium equation.  In Appendix \ref{sec:aprori}, we prove global existence of solutions with small initial data in the Weiner algebra.

\subsection*{Acknowledgements} This project was started while JLM was
on sabbatical at Duke University in the Spring of 2019.  JLM thanks
Anya Katsevich, Bob Kohn, Dio Margetis and Jonathan Weare for many
valuable conversations regarding modeling of kinetic Monte Carlo.  JGL
was supported by the National Science Foundation (NSF) grant
DMS-1812573 and the NSF grant RNMS-1107444 (KI-Net).  JL was supported
by the National Science Foundation via grant DMS-1454939.  JLM
acknowledges support from the NSF through NSF CAREER Grant
DMS-1352353.

\section{Generalized broken bond models}
\label{deriv}

 \subsection{Overview of microscopic system and its statistical mechanics}

 We will assume that the crystal surface consists of
 height columns described by
 \begin{equation}
 \label{hdef}
 h:=(h_i)_{i=1,\ldots N}
 \end{equation}
 with screw-periodic boundary conditions in the form
\begin{equation*}
	h_{i+N}=h_i+ \zeta  N \quad \forall i~,
\end{equation*}
where $\zeta$ is the average slope and each $h_i \in a \mathbb{Z}$ lives on a lattice with discrete height jumps given by some value $a \in \mathbb{R}$.  Below, on the whole crystal starting in Section \ref{KMC} we generically take $\zeta=0$, $a=1$,  however we will continue in the general setting here since we locally approximate the non-equilibrium dynamics below as equilibrium dynamics around a mean of fixed slope.  Locally, this can be seen to be an equilibrium for the dynamics we propose.  A schematic of the microscopic dynamics is given in Figure \ref{fig:cartoon}.

\begin{figure}
  \centering
  \includegraphics[width=8cm]{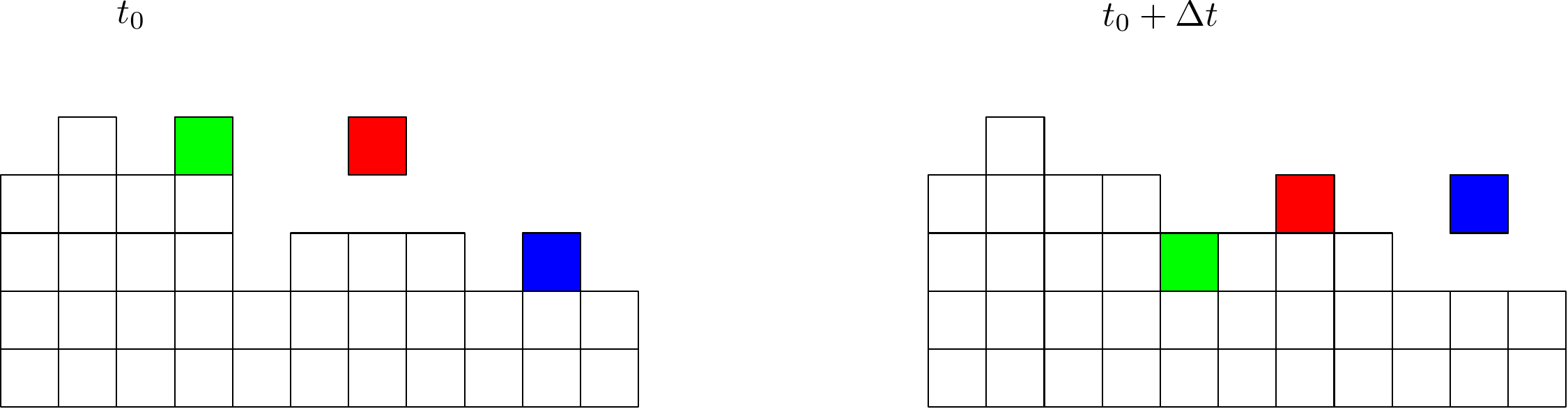}
  \caption{A cartoon of atoms in a crystal lattice in an initial configuration on the left and moving to a  likely new configuration on the right stemming from an atom jumping from one site to another (green), depositing on the sample (red) and evaporating from the sample (blue).}
  \label{fig:cartoon}
\end{figure}

The surface free energy of the general system equals
\begin{equation*}
E(T,h)=E_b+E_s~,
\end{equation*}
where $E_b$ is the bulk contribution, namely,
\begin{equation}
E_b=-2\gamma \frac{N}{a} +  \frac{\gamma \zeta} {2a} N,
\end{equation}
and $E_s$ is the relative energy (or the surface contribution), viz., 
\begin{equation*}
E_s=\frac{\gamma}{2a}\sum_{i=1}^N |h_i-h_{i-1}|	
\end{equation*}
by simply bond counting \cite{KDM,LS2018} with $\gamma$ is a
proportionality constant.  Instead of the simplistic model above,
taking $V(s) = |s|^p$ as the interaction potential, we assume a
generalized relative energy is given by
\begin{equation*}
E_s=\frac{\gamma}{2a}\sum_{i=1}^N V (h_i-h_{i-1}).	
\end{equation*}
Examples include the quadratic behavior $p=2$ due to elastic interactions \cite{Spohn1993} and the SOS bond
counting model $p=1$.  For simplicity, we will focus here on the $p=2$
case, as the $p=1$ model has similar structure but will require
further technical calculations.  From above, it is easy to see that in
fact taking $z_i = h_{i} - h_{i-1}$, we have that the surface energy
is actually a function of local slope, not of local height,
i.e. $E(T,h) = E(T,z)$ where $z :=(z_i)_{i=1,\ldots N}$.

In terms of the larger statistical mechanics picture, the partition
function over all possible states $h$ is then given by
\[
Z(T,h,N): = \sum_s e^{-\beta E_s}.
\]
(need introduce $\beta$ here...)
Hence, we easily see that we can write the Helmholtz free energy as
\[
F := -k_B T \ln Z.
\]
The Gibbs free energy 
\[
G(T,q,N) := \min_h \{  F(T,h,N) + q h \}
\]
can be seen as the Legendre transform of $F$ with respect to $h$.  The thermodynamic
potential is also the Legendre transformation of $F$ with respect to $N$,
\[
\Omega (T,h,\mu) := \min_N \{  F(T,h,N) - \mu N \}.
\]
We will see that the out of equilibrium dynamics for the system will
in particular follow by take expectation values with respect to Gibbs
measure conditioned around a local mean (or the most probable local
state).

\begin{enumerate}[(i)]

\item On the physical scale, we have $NM$ columns of atoms, and each column consists of $O(1)$ atoms. For this physical model, we have a microscopic dynamics defined (say through the master equation); the state is then described by a distribution function as
\[
\Psi (h_1, h_2, ..., h_{MN};t).
\]

\item We will partition the domain into $N$ boxes, each consists of $M$ columns, and we define mesoscopic variables of
\[
\bar{h}_k, \bar{z}_k, k = 1, \dots,N,
\]
they are the average height and average slope of each box, so that for example,
\[
\bar{z}_k = \frac{1}{(M-1)} (h_{(k+1)M} - h_{kM + 1}), \ k = 1, \dots, N.
\]
For the continuum limit, we will regard these functions as evaluations of a continuous analog at grid points, thus
\[
\bar{z}_k = \bar{z}( k/N )
\]
for $\bar{z}$ a function defined on $[0, 1]$ (the notation here is overloaded).  Note that if $M$ is large, we can view these averaged quantities taking continuous values.

\smallskip 
\item[] Connecting to the microscopic variables, we will assume that

\item $\Psi$ is given by a tensor product as given by the molecular
  chaos assumption as standard in kinetic theory \cite{Maxwell}
\[
 \Psi = f_1(h_1, h_2, ..., h_M) f_2(h_{M+1}, ..., h_{2M}) ... f_N(h_{(N-1)M+1} ... h_{NM}).
\]

\item Each $f_k$ is given as a Gibbs state, thus
\begin{equation}
\label{chempot}
\hspace{2cm} f_k(h) \propto \exp( \beta \mu_k \bar{h}_k - \beta V (\bar{z}_k)),
\end{equation}
where $\mu_k$  is the chemical potential and determined by shifting the mean of $z_k$ to fit a most probable state.  To derive the evolution equations for the averaged quantity $\bar{h}$ (and hence $\bar{z}$, we analyze the non-equilibrium dynamics of the generator of the microscopic process, following closely the work \cite{MW1}.
\end{enumerate}

\begin{remark}
  Let us remark that the thermodynamic setup we use for the
  hydrodynamic limit is very close in spirit with the one used by
  Smereka~\cite{Smereka}. The slight differences of the setup of
  \cite{Smereka} compared with the above framework are as follows.   (1) Instead of
  explicitly enforcing that the chemical potential $\mu_k$ being
  constant over each box containing $M$ columns.  A (local) chemical
  potential is assigned to each column in \cite{Smereka} with the
  implicit assumption that it is slowly varying (so that one can
  ``lump'' together several columns under a piecewise constant
  approximation to the chemical potential). (2) Correspondingly, the
  basic variable used in \cite{Smereka} is the difference of height
  between neighboring columns: $z_k = h_{k+1} - h_k$, while we have
  used $\bar{z}_k$ which is an averaged height difference over
  neighboring boxes of columns.  Our framework makes explicit the
  intermediate scale by grouping $M$ columns together; while it is
  implicitly assumed in \cite{Smereka}.
\end{remark}

 \subsection{Kinetic Monte Carlo atomistic model}
\label{KMC}

Our procedure of deriving the continuum theory mimics that of deriving
(generalized) hydrodynamics; the only assumption is the molecular
chaos (so that $F$ is a product measure) and local equilibrium (so
that $f_k$ is given as a generalized Gibbs state).

Let us rescale so that the system is defined on the unit interval
$[0, 1]$ with $N$ columns of atoms, and thus the lattice constant
becomes $\frac{1}{N}$.  We now denote the heights of these $N$ as a vector $h$ as in \eqref{hdef} with each $h_i \in \mathbb{Z}$, $i = 1, \ldots, N$.
Thus, before rescaling, we consider a crystal surface with width $N$.

We consider the continuum limit that $N \to \infty$ and average over windows of size $M \ll N$, but such that $M \to \infty$.
In particular, in this scaling, the
height function will converge to a $\mathcal{O}(1)$ function on $[0, 1]$.   This can be manifested by viewing each column in our model as a coarse-grained version of grouping $M$ columns together in an original physical model with
 $1 \ll M \ll N$, such that $\frac{M}{N} \to 0$ as $N \to \infty$.

Our dynamics will be specified by a continuous time Markov jump process.  The surface hopping part of the process evolves by jumps from one state $h=(h_i)_{i=1,\ldots N}$ in \eqref{hdef} to another state $J_\alpha^\gamma h$  is defined by transition
\[
h \mapsto J_\alpha^\gamma h,
\]
where
\[
J_\alpha^\gamma = J_\alpha J^\gamma\quad  \text{ for }\alpha, \gamma \in \{1,2, \cdots, N\}
\]
such that at $\tau \in \{1,2, \cdots, N\}$
\[
J_\alpha h (\tau)  :=
\begin{cases}
 h (\alpha)-1,& \tau =\alpha\\
 h (\tau), & \tau \neq \alpha
 \end{cases}
\]
and
\[
  J^\alpha h (\tau) := \begin{cases}
  h (\alpha)+1,& \tau=\alpha\\
  h (\tau),& \tau \neq\alpha.
  \end{cases}
\]
Note that the transition $h \mapsto J_\alpha^\gamma h$ preserves the mass of the crystal, $m=\sum_{\alpha \in \mathbb{T}_N} h (\alpha)$.

For any $g:\mathbb{T}_N \rightarrow \mathbb{R}$ define
 the symbols $\nabla^+ g (\alpha)$ and $\nabla^-g (\alpha)$ by
\[
\nabla^+ g (\alpha) :=  g(\alpha+1) - g(\alpha)\qquad\text{and}\qquad \nabla^- g(\alpha) := g(\alpha)-g(\alpha-1).
\]
Now that we have defined the transitions by which the crystal evolves we need to specify the rate at which those transitions occur.  To that end we first recall the \emph{coordination number} of \cite{MW1}, given by $n(\alpha)$ for $\alpha\in \mathbb{T}_N$ by
\begin{multline}\label{gcn}
n (t,\alpha) :=   \frac{1}{2} \big[ V(\nabla^+ J_\alpha  h (t,\alpha)) - V(\nabla^+ h (t,\alpha)) \\+ V(\nabla^- J_\alpha h (t,\alpha)) - V(\nabla^- h (t,\alpha)) \big].
\end{multline}
One can think of $n(\alpha)$ as the (symmetrized) energy cost associated with removing a single atom from site $\alpha$ on the crystal surface.

As seen in \cite{MW1}, we have that if $V( z) = |z|,$ which is the example considered in \cite{KDM}.  Then,
\[
n (\alpha) + 1=  \sum_{\substack{ \gamma \in\mathbb{T}_N \\ |\alpha-\gamma|=1}} \mathbf{1}_{(h (\alpha)\leq h (\gamma))}
\]
where
\[
\mathbf{1}_{(h (\alpha)\leq h (\gamma))} := \begin{cases} 1 & \text{if } h (\alpha)\leq h (\gamma)\\
0 & \text{otherwise}
\end{cases}.
\]
In words,  up to an additive constant (which amounts to a time rescaling), the coordination number is the number of neighbor bonds that need to be broken to free the atom at lattice site $\alpha$. If we suppose $V(z) = z^2.$  Then
\[
n (\alpha) - 2 = \nabla^+ h (\alpha) - \nabla^- h (\alpha),
\]
i.e., up to an additive constant, the coordination number is the discrete Laplacian of the surface at lattice site $\alpha.$

  The equilibrium probability for the surface gradients $\nabla_i^+ h (\cdot)$  is then the normalized Gibbs distribution
\begin{equation}
\label{eqprob}
\rho_N  \left(  \nabla^+ h (\cdot) \right) \propto \exp\left({- \beta \sum_{\substack{\alpha \in \mathbb{T}_N}} V(\nabla^+ h (\alpha))}\right).
\end{equation}
Note that our assumption that $V$ is symmetric obviates inclusion of terms in the sum involving $\nabla_i^-h (\cdot).$

We will assume that the atom at site $\alpha$ breaks the bonds with its nearest neighbors at a rate that is exponential in the coordination number.  Once those bonds are broken the atom chooses a neighboring site of $\alpha,$ for example  with $|\gamma -\alpha|=1$, uniformly and jumps there, i.e. $ h \mapsto J^\gamma_\alpha h .$
Since there are $2$ sites $\gamma$ with $|\gamma-\alpha|=1,$ the rate of a transition $
h \mapsto J_\alpha^\gamma h
$, $r$, is a standard adatom mobility with a given Arrhenius rates (as in \cite{KDM,MW1})
\[
r (t,\alpha) = \frac{1}{2} e^{- 2 \beta n (t,\alpha)}.
\]
As with $h$ we will occasionally omit the $t$ argument in $n$ and $r.$

We define the crystal slopes
\begin{equation}
z_i = \frac{h_{i+1}-h_i}{N^{-1}}.
\end{equation}
For a given inverse temperature parameter $\beta = \frac{1}{KT}$ with $K$ the standard Boltzmann constant, we will assume that evaporation are configured by the local geometry and hence occur with rates given by $r_{evap} = r_{evap} (z)$, while for deposition rate, $r_{dep}$, we will assume a constant rate.  More precisely, the evaporation rate function $r_{evap}$ depends on slopes at two consecutive sites and is given by
\begin{equation}
\label{eqn:rates}
r_{evap} (\beta, z_i, z_{i-1}) = \rho_{evap} e^{-\frac12 \beta N^{-p} [ V (z_i) - V (z_{i-1}) ] }.
\end{equation}
Note that the energy barrier given by $V(z) = |z|^p$ is the interaction potential (if $p=1$ this is the bond counting functional giving the adatom rates), but that here it
is rescaled proportional to $N$ (in physical terms, the energy barrier depends on the lattice parameter).\footnote{We may also scale the temperature so that a factor $N$ would arise in the exponent.}  For the rate of deposition, we simply take
\begin{equation}
\label{eqn:deprates}
r_{dep} (\beta, z_i, z_{i-1}) = \tau^{-1}_{dep} e^{-\frac12 \beta \mu },
\end{equation}
where $\mu$ represents a chemical potential difference between the reservoir and surface.  The constants $\rho_{evap}$ and $\tau^{-1}_{dep}$ will need to be scaled with $N$ below.

The above description of the evolution of the process $h$ is summarized by its generator $\mathcal{A}_N.$  Knowledge of the generator allows us as in \cite{MW1} to propose the evolution of any test function $\phi$ of the crystal surface is described by
\[
\phi (h(t,\alpha))-\phi(h(0,\alpha)) = \int_0^t \left[\mathcal{A}_N \phi \right](s,\alpha)
+ M_\phi (t,\alpha)
\]
where $M_\phi (t,\alpha)$ is a Martingale with $M_\phi (0,\alpha)=0$ and whose expectation at time $t$ (over realizations of $h$) given the history of $h$ up to time $s\leq t$ is simply its value at time $s.$  In particular, we assume that $\mathbf{E}\left[ M_\phi (t,\alpha)\right] = 0$ for all $t$ and $\alpha$ where $\mathbf{E}$ is used to denote the expectation over many realizations of the surface evolution from a particular initial profile.  For our process, the generator is
\begin{align}\label{gen}
& \mathcal{A}_N \phi (h) =  \sum_{\substack{\alpha,\gamma \in\mathbb{T}_N \\ |\alpha-\gamma|=1}} r_N(\alpha) \left( \phi (J_\alpha^\gamma h) - \phi (h)\right)  \\
& \hspace{.5cm} + \sum_{\alpha \in\mathbb{T}_N} \left[ r_{dep} (\alpha ) \left(  \phi (J^\alpha h) - \phi(h)\right)  - r_{evap} (\alpha) \left( \phi (J_\alpha h) - \phi (h)\right) \right]   . \notag
\end{align}
One can check that
\[
\langle g\,(\mathcal{A}_N \phi)\rangle_N =  \sum_{h} g\,(\mathcal{A}_N \phi)\,p_N (h) = \sum_{h} \phi \,(\mathcal{A}_N g)\,p_N (h) = \langle \phi\,(\mathcal{A}_N g)\rangle_N ,
\]
i.e. that $\mathcal{A}_N$ is self adjoint with respect to the $p_N$ weighted inner product.  The jump process defined by the rates above is reversible and ergodic with respect to $p_N.$

As we are interested in exponential mobility factors, following \cite{MW1} only one possible scaling regime arises.
% First, for any function $f: [0,\infty)\times \mathbb{T}^d_N \rightarrow \mathbb{R}$ we define the projections
%$\bar f_N: [0,\infty)\times [0,1)^d \rightarrow \mathbb{R}$ by
%\begin{equation}\label{smoothscaling}
%\bar f_N(t,x) = N^{-1}  f(N^4 t, \alpha)\qquad \text{for}\qquad  N x \in \prod_{i=1}^d \left[ \alpha_i - \frac{1}{2}, \alpha_i+\frac{1}{2}\right).
%\end{equation}
In particular, we set
\[
q = \frac{p}{p-1}
\]
and, for any function $f: [0,\infty)\times \mathbb{T}_N \rightarrow \mathbb{R},$ define the projections
$\bar f_N: [0,\infty)\times [0,1) \rightarrow \mathbb{R}$ by
\begin{equation}\label{roughscaling}
\bar f_N(t,x) = N^{-q}  f(N^{q+2} t, \alpha)\qquad \text{for}\qquad  N x \in  \left[ \alpha - \frac{1}{2}, \alpha +\frac{1}{2}\right).
\end{equation}
We have scaled the crystal extent by $N.$  Note that the scaling of time and crystal height is different than a standard $4$th order diffusion scaling.   The crystal's height is now scaled at a rate faster than $N,$ and determined by the properties of the underlying potential.  The unusual scaling of time is again determined by the requirement that the limiting equation be meaningful.  This clearly motivates the choice $p=2$, as in this case we see easily that $q=p=2$. However, if $p=1$, we see that this scaling degenerates to $q=\infty$.  Indeed, to properly prove the exponential mobility in the case of the bond counting model from \cite{KDM}, one needs either to use the analysis in \cite{LLMM}
or to reduce the temperature and hence increase $\beta$ with the system size $N$ in a manner that does not require such degeneration of the time scaling.

\begin{remark}
Much of this section is motivated by similar calculations in \cite{MW1}, however we have included the details here to in particular clearly state where the $2$nd order terms arise in the generator and how to introduce the appropriate rates to the KMC process defined in \cite{MW1} without considering such terms.  
\end{remark}

\begin{remark}
Another way to see dynamics for this process is to study it in a formal hydrodynamic limit.  For a configuration $h =(h_i)_{i=1,\ldots N}$, let us define $\hat{h}^{i} = (h_j)_{j=1,\ldots N, j \neq i}$ and $\hat{h}^{i,i+1} = (h_j)_{j=1,\ldots N, j \neq i,i+1}$.  Then, the resulting Fokker-Planck equation in the setting of adatom rates for the crystal fluctuations as well as deposition ($\tau^{-1}$) and evaporation ($\rho$) rates is then
\begin{equation}
\begin{aligned}
  \partial_t \Psi (h, \beta; t) & = \frac{D}{2} \sum_i \Bigl\{   \Psi  (h_i + 1, h_{i+1} - 1, \hat{h}^{i,i+1}, \beta; t) r (\beta, z_i - 2N, z_{i-1} + N)  \\
  &  \hspace{1cm} + \Psi  (h_i - 1, h_{i+1} +1, \hat{h}^{i,i+1}, \beta ; t) r (\beta, z_i + 2 N, z_{i-1} -  N)  \\
  & \hspace{1cm} - 2 \Psi   (h_i , h_{i+1} , \hat{h}^{i,i+1}, \beta; t)  r (\beta, z_i, z_{i-1})    \Bigr\} \\
  &  +  \rho_{evap} \sum_i \Bigl\{ \Psi ( h_i - 1, \hat{h}^i,  \beta ; t  )  r ( \beta,
  z_i +N, z_{i-1} + N )  \\
  &  \qquad \qquad- \Psi   ( h_i , \hat{h}^i ,  \beta; t  )  r ( \beta,
  z_i , z_{i-1}) \Bigr\} \\
  & + \tau_{dep}^{-1} \sum_i \Bigl\{ \Psi  ( h_i + 1, \hat{h}^i,  \beta ; t )  r_{dep} ( \beta,
  z_i - N, z_{i-1} + N )\\
  & \qquad \qquad - \Psi  ( h, \beta ; t )  r_{dep} (\beta, z_i, z_{i-1} ) \Bigr\}.
\end{aligned}
\end{equation}
\end{remark}

 \subsection{Macroscopic Dynamics}

To derive the PDE limits from the generator of the microscopic process, we follow the calculations in Section $4$ and $5$ of \cite{MW1}. For finite $\beta$, we wish to find $\eta$ to shift the mean of the distribution on each window.  In particular, we wish to find $\eta$ such that
\begin{equation}
\label{maxEndist1}
\frac{ \sum_{z \in \mathbb{Z}} e^{ -\beta    V(z)   + \eta z  }     }{ Z_{\eta}}
\end{equation}
to $ \bar{z}$ where we will condition $\bar{z}$ on each window of size $M$.  Here note the slight abuse of notation to use $z$ as its scalar component averaged on a window.

For fixed finite $\beta$, we observe that our process has only conservation law $z$ we have arrived at precisely the optimal twist distribution used in \cite{MW1}.   For $u\in \mathbb{R},$  the surface tension $\sigma_D(u)$ (\cite[Sec. 5]{Funaki}) is defined using the Legendre transformation
\begin{equation}\label{eqn:sigmaD}
\sigma_D(u) =\sup_{\eta\in \mathbb{R} }\left\{ \eta u  - \log Z_{\eta} \right\}
\end{equation}
with
\[
Z_{\eta} =
 {\sum_{z\in \mathbb{Z} } e^{-\beta V(z) +  \eta z }}.
 \]
 Note, for $p=2$, this just becomes
 \[
Z_{\eta} \sim
 {\sum_{z\in \mathbb{Z} } e^{-\beta  \left( z - \frac{\eta}{2 \beta} \right)^2 }}.
 \]
 As a result, we can construct the surface tension through similar arguments. In particular, we let
 \[
u = \left[\nabla_\eta \log Z_\eta \right]( \sigma'_D(u)) =  \frac{\sum_{z\in \mathbb{Z}}  z\, e^{-\beta V(z) +   \nabla \sigma_D z}}
 {\sum_{z\in \mathbb{Z}} e^{-\beta V(z) + \nabla \sigma_D z }}
 \]
i.e., take $\eta= \sigma'(u)$ in the pair $\eta z$ then the mean value of $z$ under the distribution  \[
 \frac{ e^{-\beta V(z) +\eta z}}{Z_{\eta}}
 \]
 is $u$ \cite[Sec. 5]{Funaki}.  
    In other words, the chemical potential $\mu_k$ as in \eqref{chempot} should be seen to converge to $-\pd_x \sigma'(u)=-\sigma''(u)\pd_x u=- \sigma''_D (h_x) h_{xx}$ in \eqref{eqn:smereka_disc} below.  However, due to the scaling in $N$, the PDE limit will require that we characterize the behavior of $N^{-1} \nabla \sigma_D(  N^{q-1} u)$.  More precisely we need to consider the limit  $\kappa^{1-p}\nabla \sigma_D(\kappa u)$ as  $\kappa$ grows very large.   For $p=2$, it is clear that the limit of $\kappa^{-1}\nabla \sigma_D(\kappa u)$
exists and that
\begin{equation}\label{barsigma}
 \lim_{\kappa \rightarrow \infty} \kappa^{-1}\nabla \sigma_D(\kappa u) = 2 \beta u.
\end{equation}

Continuing along, by partitioning $\mathbb{T}$ into small but macroscopic sets, let $\delta = M/N$ with $N^{-1}\ll \delta \ll 1$ and define the sets
\[
S_k = \mathbb{T} \cap  \delta\left[
k , k+1 \right) .
\]
Note, the volume of each (non-empty) set $S_k$ is the same (and equal to $\delta$).  Hence, following the analysis of \cite{MW1}, Section $6.2$, we observe that taking the expectation of the generator with respect to the local Gibbs measure limits to the
\begin{align*}
\varphi^\delta_{N,k}(t) &- \varphi^\delta_{N,k}(0)
   \approx \delta^{-1}  \frac12 \int_0^t \int_{S_k}
 \Delta \left[
e^{ -\beta \text{div}\left[\nabla V(\nabla h(s,x))\right]}  \right]dx ds  \\
& \hspace{0.0cm} +  N^2 C_N  (\rho_{evap}, \tau^{-1}_{dep})  \delta^{-1}  \frac12 \int_0^t \int_{S_k}
 1 -  \left[
e^{ -\beta \text{div}\left[\nabla V(\nabla h(s,x))\right]}  \right]dx ds
    \end{align*}
where here the operators $\Delta, \text{div}$ are the properly rescaled discrete differential operators acting on a lattice of uniform spacing $1/N$.

As the KMC models are inherently discrete, the model that arises once $\rho_{evap}$ and $\tau^{-1}_{dep}$ have been scaled appropriately to balance the time re-scaling is of the discrete form
\begin{eqnarray}
\label{eqn:smereka_disc}
\frac{d h_k}{dt} = \frac{\alpha}2 \left[  e^{\beta \mu_{k-1}} - 2 e^{\beta \mu_{k}} + e^{\beta \mu_{k+1} }\right]+ C  (1 -  e^{\beta \mu_{k}}),
\end{eqnarray}
provided the evaporation and deposition rates, $\rho_{evap}, \tau^{-1}_{dep}$, scale appropriately with $N$.  Note, this scaling makes sense physically, namely that the evaporation and deposition must be slowed as the system size scales up in order for surface fluctuations and epitaxial properties to balance.

For $V(z) = |z|^2$, we have that $\mu_k =  -2( \Delta_N h)_k$.   The discrete system is identical to the form derived by Smereka \cite{Smereka}. The $N \to \infty$ limit can then be seen to be of the form \eqref{pde}.  Namely, for large $N$ and small $\delta,$ $\varphi_{N,k}^\delta(t) \approx h(t,k\delta)$ we obtain
\begin{align*}
  h(t,x) - h(0,x)
  &= \delta^{-1}  (2)^{-1} \int_0^t \int_{S_k}
    \Delta \left[
    e^{ -\beta \text{div}\left[\nabla V(\nabla h(s,x))\right]}  \right] \\
  &\qquad   + C \left[ 1- e^{ -\beta \text{div}\left[\nabla V(\nabla h(s,x))\right]}  \right]   dx ds .
    \end{align*}

Physically, the above approach is motivated by the work of Krug et al \cite{KDM} and Smereka \cite{Smereka} both of whom derived PDE limits using closure equations and moments of the Gibbs measure.  We can take a maximal entropy approach to simplify some of the presentation in a manner that works with the derivation via the generator by clearly predicting what measure to take expectations with respect to in order to capture non-equilibrium dynamics locally.  We will take $p[h,z,E]$ the probability of configuration in terms of slope $z$, and take the entropy $p[h,z] \log p[h,z]$.  Given a distribution $p$, take the functions:
\[
F_z [p] \to \bar{z}, \  F_h [p] \to \bar{h},
\]
the maps from the distribution $p$ to the average height $\bar{h}$, average slope $\bar{z}$ respectively.   Via a consistency condition, must we have that $\zeta = -\nabla \eta$ and hence we wish to condition the non-equilibrium dynamics on selecting the most likely $z$ configuration on each window.

So, the grand canonical ensemble will be of the form
\begin{equation*}
p \sim e^{  \beta \mu_h \cdot \bar{h}  - \beta E_s},
\end{equation*}
where we now need to substitute $p$ as in the master equation.  The local Gibbs measure in particular is of the form
\begin{equation*}
p \sim e^{ -\beta( E_s - \mu h)},
\end{equation*}
where $\mu = \frac{\delta F}{\delta h}$ is the chemical potential such that $\mu$ shifts the mean of the distribution to the most probable state.  Above, we have observed that in the scaling limit resulting in exponential mobility we have
$$\mu = - \partial \partial_{z}   V(z) \sim - \diver (V' (\nabla h)),$$
as the Helmholtz free energy is of the form
\begin{equation*}
F ( z ) \sim  V(z)
\end{equation*}
and we observed that the surface tension term $\nabla \sigma_D (\nabla h) = V' (\nabla h) $ is the shift of the mean in $z$.  Hence, the chemical potential must arise from using summation by parts to move a derivative over onto $h$.

\begin{remark}
 The methodology of window averaging we present here is essentially identical to that of \cite{MW1}, though we have attempted to more clearly connect the methods to previous works as well as to more standard statistical physics conventions for the reader's convenience.  
\end{remark}

The goal of the remaining sections will be to establish some analytical results for a continuum version of the model \eqref{pde}, and in particular to compare and contrast dynamics with the purely $4$th order model.

\section{Global  Weak Solutions Positive almost everywhere}
\label{weak}

In this section we prove the global weak solution to \eqref{pde} by considering another degenerate parabolic equation.
Define
$$u:= e^{-\Delta h}.$$
Then \eqref{pde} can be formally recast as
\begin{equation}\label{PDE}
\partial_t u =-u\Delta[\Delta u+ (1- u)].
\end{equation}
If $u>0$ for all time then \eqref{PDE} and \eqref{pde} are equivalent
to each other rigorously.  We investigate the global weak solution for
\eqref{PDE} in one dimension with periodic boundary condition in
$\I=\mathbb{R}/ \mathbb{Z}$.  The periodic boundary condition is a
natural one for the screw-periodic $h$ as discussed in the previous
section.   As many of the tools we present here
  connect well to literature on $1d$ thin film models, we restrict our
  attention here to one dimensional models to avoid the dimensional restriction for embedding theorem.   

We are going to prove there exists a  global weak solution to
\eqref{PDE}, which
 is positive almost everywhere. In the other words, the set $\{(t,x); u=0\}$, which corresponds to the singular points for $\{(t,x); \Delta h=+\8\}$, has Lebesgue measure zero; see Theorem \ref{global_th} and the proof in Section \ref{sec3.3}. 
The asymptotic behavior of $u(t,x)$ as time goes to infinity will also be proved in Theorem \ref{long2}.

\subsection*{Notations } 
In the following, using standard notations for Sobolev spaces,
we denote
\begin{equation}
H^{k} (\I):=\{u(x)\in H_{\text{loc}}^k(\mathbb{R}); { \,u(x+1)=u(x) \,\,a.e.\,x\in \mathbb{R} } \},
\end{equation}
with standard inner product in $H^k$
 and when $k=0$, we denote it as $L^2 (\I)$.

\subsection{Formal observations and existence Result}
Denote the first functional $F$ as
\begin{equation}
  F(u):= \int_\I u \d x.
\end{equation}
 Denote the second functional $E$ as
\begin{equation}
  E(u):= \int_\I (\pd_{xx} u)^2 + (\pd_x u)^2\d x.
\end{equation}
We first give key observations which inspire us to prove  the regularities and positivity of solutions later. \\
{\bf Observation 1.} We have the following lower order energy dissipation law
\begin{align}
  \frac{\d F(u)}{\d t} & = \int_\I  \partial_t u  \d x = \int_\I -u (\pd_x^4 u-\pd_{xx}u) \d x \\
  & = -\int_\I (\pd_{xx} u)^2 + (\pd_x u)^2 \d x = - E(u) \leq 0 \nonumber.
\end{align}
This also shows the relation between $F$ and $E$, which is the key point to study the asymptotic behavior of  solutions.\\
{\bf Observation 2.} We have the following higher order energy dissipation law
\begin{align}
  \frac{\d E(u)}{\d t} &= 2\int_\I \pd_{xx}u \pd_{xx}\partial_t u + \pd_x u \pd_{xt} u\d x\\
  &=\int_\I 2 (\pdf u -\pd^2_x u)\partial_t u  \d x = -2 \int_\I u (\pdf u-\pd^2_x u)^2 \d x  \leq 0\nonumber
\end{align}
if $u\geq 0$.\\
{\bf Observation 3.} We have the following heuristic estimate to obtain the lower bound of solution $u$.
\begin{align}
  \frac{\d}{\d t} \int_\I \ln u \d x =0,
\end{align}
if $u>0$.

Taking into account Observation 3, although we can prove the measure of $\{(t,x);u(t,x)=0\}$ is zero, we still have no regularity information for this degenerate set. To avoid the difficulty when $u=0$, following the idea of \textsc{Bernis and Friedman} \cite{Friedman1990}, we use a regularized method  to first prove the existence and strict positivity for regularized solution $\uv$ to a properly modified equation below, then take limit $\varepsilon\to 0$. For $0<\alpha<1$, the regularization of \eqref{PDE} we consider is
\begin{equation}\label{equv}
 \left\{
     \begin{array}{ll}
       \dpstyle{\uvt=-\frac{\uv^{1+\alpha}}{\uv^\alpha+ \varepsilon^\alpha}(\pd_x^4 \uv-\pd_x^2 \uv),} & \text{ for }t\in[0,T],\,x\in \I; \\
       \dpstyle{\uv(0,x)=u_0(x)+\varepsilon,} & \text{ for }x\in \I.
     \end{array}
   \right.
\end{equation}
We will show in \eqref{n52}  that $\uv$ has a lower bound $\varepsilon$ for all $t\in[0,T]$; see Section \ref{sec3.3.2}.
Therefore the regularized problem \eqref{equv} is nondegenerate for fixed $\varepsilon$. 
The  existence of the regularized problem is also stated in  \cite{Friedman1990}. We also refer to \cite{GLL2017} for the uniqueness of the solution to a similar regularized problem.
 We point out that the non-degenerate regularized term is important to the positivity of the global weak solution.

Since the lower bound for $\uv$ depends on $\varepsilon$, we can only prove the limit solution $u$ is positive almost everywhere. Therefore, we need to define a set
\begin{equation}\label{PT}
  P_T:=(0,T)\times \I \backslash \{(t,x);u(t,x)=0\},
\end{equation}
which is an open set and we can
define a distribution on $P_T$. From now on, $c$ will be a generic constant whose value may change from line to line.

First we give the definition of weak solution to PDE \eqref{PDE}.
\begin{defn}
  For any $T>0$, we call a non-negative function $u(t,x)$ with regularities
  \begin{equation}\label{723_1}
    u\in \Lty([0,T];\Hper^2(\I)),\quad u(\pd_x^4 u - \pd_x^2 u)\in L^2(P_T),
  \end{equation}
  \begin{equation}\label{723_3}
    \partial_t u \in L^2([0,T];\Lper^2(\I)),\quad u\in C([0,T];\Hper^1(\I)),
  \end{equation}
   a weak solution to PDE \eqref{PDE} with initial data $u(0,x)=u_0(x)$ if

  \begin{enumerate}

    \item for any function $\phi\in C^\infty([0,T]\times\I)$,  $u$ satisfies
    \begin{equation}\label{eq01}
      \int_0^T\int_\I \phi \partial_t u  \ud x \ud t+\int\int_{P_T}\phi u(\pd_x^4 u - \pd_x^2 u)\ud x \ud t =0;
    \end{equation}

    \item the following  first energy-dissipation inequality holds
    \begin{equation}\label{E01}
      E(u(T,\cdot))+\int\int_{P_T}2u(\pd_x^4 u - \pd_x^2 u)^2 \ud x \ud t \leq E(u(0,\cdot)).
    \end{equation}

    \item the following  second energy-dissipation inequality holds
    \begin{equation}\label{F01}
      F(u(T,\cdot))+\int_0^TE(u(t,\cdot)) \ud t \leq F(u(0,\cdot)).
    \end{equation}

  \end{enumerate}

\end{defn}

%
%We now state the main result the global existence of weak solution to \eqref{PDE} as follows.
%\begin{theorem}\label{global_th}
%  For any $T>0$, assume initial data $u_0\in \Hper^2([0,1])$, $\int_\I \ln(u_0) \ud h =m_0<+\infty$ and $u_0\geq0$. Then there exists a global non-negative weak solution to PDE \eqref{PDE} with initial data $u_0$. Besides, we have
%  \begin{equation}\label{u+}
%    u(t,x)>0, \text{  for }\ale (t,x)\in[0,T]\times[0,1].
%  \end{equation}
%\end{theorem}

{We now state the main result the global existence of weak solution to \eqref{PDE} as follows.
\begin{theorem}\label{global_th}
  For any $T>0$, assume initial data $u_0\in \Hper^2(\I)$, with
  \[\int_\I \ln(u_0) \ud x =:m_0<+\infty, \ \ u_0\geq0.  \]
   Then there exists a global non-negative weak solution to PDE \eqref{PDE} with initial data $u(0,x)=u_0(x)$. Besides, we have
  \begin{equation}\label{u+}
    u(t,x)>0  \text{  for }\ale (t,x)\in[0,T]\times\I.
  \end{equation}
\end{theorem}
}

We will use an approximation method to obtain the global existence in Theorem \ref{global_th}. This method is proposed by \cite{Friedman1990} to study a nonlinear degenerate parabolic equation. 

{
\begin{remark}
  The regularized method for studying the 4th order degenerate problem
  is first introduced in \textsc{Bernis and Friedman}
  \cite{Friedman1990}. There are however some technical difficulties
  to overcome in applying this general method to our problem
  \eqref{PDE}. In particular, when taking limit for the regularization
  constant $\varepsilon\to 0$, we need to carefully deal with the set
  $\{(t,x); u\geq 0\}$ by dividing it into several subsets (see Lemma \ref{claim3.5}) and prove
  the Lebesgue measure of $\{(t,x); u= 0\}$ is zero (see Section \ref{sec3.3.2}). 
\end{remark}}

\subsection{Global positive solution to a regularized problem}\label{sec3.1} In this section, we will study key a-priori estimates for the regularized solution $\uv$ and obtain the lower bound of regularized solution $\uv$, which depends on $\varepsilon$.

First we give the definition of weak solution with energy identities to regularized problem \eqref{equv}.
\begin{defn}\label{defnuv}
  For any fixed $\varepsilon>0,\,T>0$, we call a non-negative function $\uv(t,h)$ with regularities
  \begin{equation}
    \uv\in \Lty([0,T];{  \Hper^2}(\I)),\quad \frac{\uv^{1+\alpha}}{\uv^\alpha+ \varepsilon^\alpha}(\pd_x^4 \uv-\pd_x^2\uv)\in L^2(0,T;{ \Lper^2}(\I)),
  \end{equation}
  \begin{equation}
    \uvt\in L^2([0,T];{  \Lper^2}(\I)),\quad \uv\in C([0,T];\Hper^1(\I)),
  \end{equation}
   weak solution to regularized problem \eqref{equv} if
\begin{enumerate}
    \item for any function $\phi\in C^\infty([0,T]\times\I)$, $\uv$ satisfies
    \begin{equation}\label{eqv01}
      \int_0^T\int_\I \phi \uvt \ud x \ud t+\int_0^T\int_\I \phi\frac{\uv^{1+\alpha}}{\uv^\alpha+ \varepsilon^\alpha}(\pd_x^4 \uv-\pd_x^2 \uv) \ud x \ud t =0,
    \end{equation}
    \item the following first energy-dissipation equality holds
    \begin{equation}\label{Ev01}
      E(\uv(T,\cdot))+2\int_0^T\int_\I \frac{\uv^{1+\alpha}}{\uv^\alpha + \veps}(\pd_x^4 \uv-\pd_x^2 \uv)^2 \ud x \ud t =E(\uv(0,\cdot)),
    \end{equation}
    { \item the following second energy-dissipation equality holds
    \begin{equation}\label{Fv01}
      F_\varepsilon(\uv(T,\cdot))+\int_0^TE(\uv(t,\cdot)) \ud t =F_\varepsilon(\uv(0,\cdot)),
    \end{equation}
    where
    $F_\varepsilon(\uv):=\int_\I \left( \frac{\varepsilon^\alpha}{1-\alpha} \frac{1}{u^{\alpha-1}}+ \uv \right) \ud x$
    is a perturbed version of $F$.  }
  \end{enumerate}
\end{defn}

The existence of global positive solution to \eqref{equv} defined above will be proved by collecting the key a priori estimates in Section \ref{sec3.3.1} and validation of the a priori assumption in Section \ref{sec3.3.2}.

First we state the key lemma connecting norm of second derivative to minimum of $u$.
\begin{lemma}\label{lem-min}
  For any function $u$ such that $u\in H^2([0,1]),$ assume that $u$ achieves its minimal value $u_{\min}$ at $x^\star$, i.e. $u_{\min}=u(x^\star)$. Then, we have
  \begin{equation}\label{lem2eq}
     u(x)-u_{\min}  \leq \frac{2}{3} \|u_{xx}\|_{L^2([0,1])}\vert x-x^\star \vert^{\frac{3}{2}}, { \text{ for any }x\in[0,1]}.
  \end{equation}
\end{lemma}
\begin{proof}
  Since $u_{xx}\in L^2([0,1]),$ $u_x$ is continuous. Hence by $u_{\min}=u(x^\star)$, we have $u_x (x^\star)=0$ and
  \begin{equation}
    \vert u_x(x)\vert=\vert \int_{x^\star}^x u_{xx}(s) \ud s \vert \leq \vert x-x^\star \vert^{\frac{1}{2}} \|u_{xx}\|_{L^2([0,1])}, \text{ for any }x\in[0,1].
  \end{equation}
  Hence we have
  \begin{align*}
    \vert u(x)-u_{\min}\vert&\leq \left|\int_{x^{\star}}^x \vert s-x^\star \vert^{\frac{1}{2}} \|u_{xx}\|_{L^2([0,1])} \ud s \right| \\
    &\leq \frac{2}{3}\vert x-x^\star \vert^{\frac{3}{2}} \|u_{xx}\|_{L^2([0,1])}.
  \end{align*}
\end{proof}

\subsubsection{A-priori estimates and energy identities under a-priori assumption $\uv>0$. }\label{sec3.3.1} In this section we will prove the lower order and higher order a priori estimates under a-priori assumption $\uv>0$. The a-priori assumption will be verified in Section \ref{sec3.3.2}.

Step 1. Higher order estimate.
Multiplying \eqref{equv} by $\pd_x^4 \uv-\pd_x^2 \uv$ gives
\begin{align}\label{tm34}
  \frac{1}{2}\frac{\ud}{\ud t}\int_\I (\pd_{x}^2 \uv)^2 & + (\pd_x \uv)^2 \ud x= \int_\I (\pd_x^4 \uv-\pd_x^2 \uv) \uvt \ud x \\
  &  = -\int_\I \frac{\uv^{1+\alpha}}{\uv^\alpha+ \varepsilon^\alpha}(\pd_x^4 \uv-\pd_x^2 \uv)^2\ud x\leq 0.  \notag
\end{align}
This gives
\begin{equation}\label{n41}
2E(\uv)+ \int_0^T \int_\I \frac{\uv^{1+\alpha}}{\uv^\alpha+ \varepsilon^\alpha}(\pd_x^4 \uv-\pd_x^2 \uv)^2\ud x\leq 2E(\uv(0)). 
\end{equation}
Thus we obtain, for any $T>0$,
\begin{equation}\label{high}
  \|(\uv)_{xx}\|_{\Lty([0,T];L^2(\I))}\leq  \sqrt{2}E_0^{\frac{1}{2}},
\end{equation}
where $E_0:= E(u_0).$

Step 2. Lower order estimate.  We require the \textit{a-priori} assumption $\uv>0$.
Multiplying \eqref{equv} by $\frac{\uv^{\alpha}+ \varepsilon^{\alpha}}{\uv^{\alpha}}$, we have
\begin{equation}\label{add18_1}
\begin{aligned}
&\frac{\ud}{\ud t}\int_\I \frac{\varepsilon^\alpha}{1-\alpha} \frac{1}{\uv^{\alpha-1}}+ \uv\ud x\\
=&\int_\I -\uv (\pd_x^4 \uv-\pd_x^2 \uv)  \ud x=\int_\I -(\pd_{x}^2 \uv)^2-(\pd_x \uv)^2 \ud x=-E(\uv)\leq 0,
\end{aligned}
\end{equation}
which implies for $\alpha<1$
\begin{align*}
&\int_\I \uv\ud x  \leq
 \int_\I \uv(0)\ud x +c \frac{\veps^\alpha}{\veps^{\alpha-1}}\leq c(u_0), \text{ for any }t\in[0,T],
\end{align*}
where we used $u(0)\geq 0$ and thus $\uv(0)\geq \veps$.
Here and in the remaining of this section, $c(u_0)$ will be a positive constant depending only on $u_0$.

Hence
from
$$(\uv)_{\min}(t)\leq \int_\I \uv\ud x  \leq c(u_0) \text{ for any }t\in[0,T] $$
and Lemma \ref{lem-min}
 we have
\begin{align}\label{supnorm}
 \|\uv\|_{L^\8(0,T;L^\8(\I))}\leq c E_0^{\frac{1}{2}} + c(u_0).
\end{align}

Combining Step 1 and Step 2, we have
\begin{equation}\label{es1}
  \|\uv\|_{\Lty([0,T];H_{per}^2[\I])}\leq  C(u_0).
\end{equation}
Moreover, from \eqref{tm34}, we also have
\begin{equation}
  \frac{1}{2}\frac{\ud}{\ud t}\int_\I (\pd_{x}^2 \uv+ \pd_x \uv)^2 \ud x= -\int_\I \frac{\uv^\alpha+ \varepsilon^\alpha}{\uv^{1+\alpha}} \uvt^2\ud x .
\end{equation}
This, together with $ \|\uv\|_{L^\8(0,T;L^\8(0,1))}\leq c E_0 + c(u_0)$, also gives uniform bound of $\pd_t \uv$
\begin{equation}
\frac{1}{c E_0 + c(u_0)} \int_0^T \int_\I  \pd_t \uv^2 \ud x \ud t \leq \frac12 E_0.
\end{equation}
Thus we have
\begin{equation}\label{es2}
\uvt \in L^2(0,T;L^2(\I)).
\end{equation}
By \cite[Theorem 4, p.~288]{Evans1998} whose proof fits also for periodic function,  we also know
\begin{equation}
\uv\in C([0,T];H^1(\I))\hookrightarrow C([0,T]\times\I).\nn
\end{equation}

  The two energy dissipation identities in Definition \ref{defnuv} follow from \eqref{tm34} and \eqref{add18_1} directly.

\subsubsection{Verify the a-priori assumption}\label{sec3.3.2}In this section, we verify the a-priori assumption $\uv>0$ by proving the lower bound of $\uv$.

Multiplying $-\frac{\uv^\alpha+ \varepsilon^\alpha}{\uv^{1+\alpha}} $ to \eqref{equv}, we obtain  the conservation law
\begin{equation}\label{conser2}
\frac{\ud}{\ud t} \int_\I \big( \frac{\veps^\alpha}{\alpha}\frac{1}{\uv^{\alpha}}-\ln \uv \big) \ud x=0.
\end{equation}
Therefore due to \eqref{supnorm},
\begin{equation}\label{49}
\frac{\veps^\alpha}{\alpha} \int_\I \frac{1}{\uv^\alpha} \ud x \leq \frac{\veps^\alpha}{\alpha} \int_\I \frac{1}{\uv^\alpha(0)} \ud x-m_0 + \int_\I \ln \uv \ud x \leq c(u_0),
\end{equation}
where we used $\int_\I \ln u_0 \ud x = m_0$ and $\uv(0)\geq \veps^{\frac{1}{\alpha}}.$
Assume for any $t\in[0,T]$, $\uv(t,\cdot)$ achieves its minimum $(\uv)_{\min}(t)$ at some point.
Notice from Lemma \ref{lem-min},
\begin{equation}
\frac23\|\pd_x^2\uv(t)\|_{L^2}|x-x^*|^{3/2}+ (\uv)_{\min}(t) \geq \uv(t) \text{ for all }t \in [0,T].
\end{equation}
Hence from  \eqref{49}  we have
\begin{align*}
{c(u_0)} \geq &~ \frac{\veps^\alpha}{\alpha} \int_\I \frac{1}{\big[ \frac23\|\pd_x^2\uv(t)\|_{L^2}|x-x^*|^{3/2}+ (\uv)_{\min}(t)\big]^{\alpha}} \ud x\\
{( \text{ from } \eqref{high})}\atop {\geq} & \frac{\veps^\alpha c}{\alpha}  \int_0^{1/2} \frac{1}{\big[ \sqrt{E_0}|x|^{3/2}+ (\uv)_{\min}(t)\big]^{\alpha}} \ud x\\
\geq & ~\frac{1}{(\uv)_{\min}^{\alpha-\frac23}}  \frac{\veps^\alpha c}{\alpha} \int_0^{\frac{1}{2(\uv)_{\min}^{\frac23}}} \frac{1}{\big[ \sqrt{E_0}|y|^{3/2}+ 1\big]^{\alpha}} \ud y\\
\geq & ~\frac{c \veps^\alpha}{\alpha} \frac{1}{(\uv)_{\min}^{\alpha}}.
\end{align*}
Now since $0<\alpha<1$, standard calculus shows that
\begin{equation}\label{n52}
(\uv)_{\min}(t)\geq c(u_0)\veps \,\text{ for all }t\in[0,T].
\end{equation}

\subsection{Global solution to original equation}\label{sec3.3}
This section is devoted to obtaining the global solution to original equation by taking the limit in the regularized problem \eqref{equv}. The proof of Theorem \ref{global_th} will be the collections of the following subsections.

\subsubsection{Convergence of $\uv$ when taking limit $\veps\to 0$}
Assume $\uv$ is the weak solution to \eqref{equv} whose existence is stated by \cite{Friedman1990} after collecting the key a priori estimates in Section \ref{sec3.3.1} and validation of the a priori assumption in Section \ref{sec3.3.2}.
From \eqref{es1} and \eqref{es2},  as $\varepsilon\rightarrow 0,$  we can use Lions-Aubin's compactness lemma for $\uv$ to show that there exist a subsequence of $\uv$ (still denoted by $\uv$) and $u$ such that
\begin{equation}\label{strong}
  \uv\rightarrow u, \text{ in }\Lty([0,T];\Hper^1(\I)),
\end{equation}
which gives
\begin{equation}\label{strong1}
  \uv \rightarrow u, \quad \ale (t,x)\in[0,T]\times \I.
\end{equation}
Again from \eqref{es1} and \eqref{es2}, we have
\begin{equation}\label{weak1}
  \uv{\stackrel{\star}{ \rightharpoonup}} u \quad \text{in }\Lty([0,T];\Hper^2(\I)),
\end{equation}
and
\begin{equation}\label{weak2}
  \uvt \rightharpoonup \partial_t u  \quad \text{in }L^2([0,T];\Lper^2(\I)),
\end{equation}
which imply that
\begin{equation}\label{weak3}
  u\in \Lty([0,T];\Hper^2(\I)),\quad  \partial_t u  \in L^2([0,T];\Lper^2(\I)).
\end{equation}
In fact, by \cite[Theorem 4, p.~288]{Evans1998} whose proof fits also for periodic function, we  know
\begin{equation}
u\in C([0,T];\Hper^1(\I))\hookrightarrow C([0,T]\times \I).\nn
\end{equation}

\subsubsection{Estimate for the measure of $\{(t,x);u=0\}$}\label{sec3.3.2}
Now we use the conservation law for $\uv$ to estimate the measure  $\text{meas}\{(t,x);u=0\}$.

From \eqref{conser2} we have for any $\delta>0,$
\begin{equation}
\text{meas}\{(t,x);\uv<\delta\}(-\ln \delta) \leq \int_0^T \int_\I -\ln\uv \ud x \ud t \leq c(u_0) T,
\end{equation}
which implies
\begin{equation}\label{mes}
\text{meas}\{(t,x);\uv<\delta\} \leq \frac{c(u_0) T}{-\ln \delta}.
\end{equation}
Therefore by \eqref{strong1}, we have
\begin{equation}
\text{meas}\{(t,x);u=0\}=\lim_{n\to \8}\text{meas}\{ (t,x); \uv<\frac{1}{n} \}=\lim_{n\to \8}\frac{c(u_0)T}{\ln n}=0.
\end{equation}

\subsubsection{Proof of Theorem \ref{global_th} by taking limit $\veps\to 0$ in \eqref{weak2}}
 Recall $\uv$ is a weak solution of \eqref{equv} satisfying \eqref{eqv01}. We want to pass to the limit for $\uv$ in \eqref{eqv01} as $\varepsilon \rightarrow 0.$ From \eqref{weak2}, the first term in \eqref{eqv01} becomes
\begin{equation}
  \int_0^T\int_\I \phi \uvt \ud x \ud t \rightarrow  \int_0^T\int_\I \phi \partial_t u  \ud x \ud t .
\end{equation}
The limit of the second term in \eqref{eqv01} is given by the following lemma. With the lemma below, one can take limit in \eqref{eqv01} and obtain \eqref{eq01}. The regularity \eqref{723_1} follows from \eqref{weak3} and \eqref{tmbd} in the proof of Lemma \ref{claim3.5}.

\begin{lemma}\label{claim3.5}
  For $P_T$ defined in \eqref{PT} and any function $\phi\in C^\infty([0,T]\times\I),$ we have
  \begin{equation}\label{claim}
     \int_0^T\int_\I \phi \frac{\uv^{1+\alpha}}{\uv^\alpha+ \varepsilon^\alpha}(\pd_x^4 \uv- \pd_x^2 \uv )\ud x \ud t \rightarrow  \int\int_{P_T} \phi u (\pd_x^4 u- \pd_x^2 u) \ud x \ud t,
  \end{equation}
  as $\varepsilon\rightarrow 0.$
\end{lemma}
\begin{proof}
First, for any fixed $0\leq \delta <1$ small enough, from \eqref{strong}, we know there exist a constant $K_1>0$ large enough and a subsequence $\uvk$  such that
\begin{equation}\label{del1}
  \|\uvk-u\|_{\Lty([0,T]\times \I)}\leq \frac{\delta}{2},\text{  for }k>K_1.
\end{equation}
Denote
\begin{align}
  & \Dld:=\{x\in [0,1];\,0\leq u(t,x)\leq \delta\}, \nn \\
  & \Dgd:=\{x\in [0,1];\,u(t,x)> \delta\}. \nn
\end{align}
The left-hand-side of \eqref{claim} becomes
\begin{equation}\label{n64}
\begin{aligned}
  I:&=\int_0^T\int_\I \phi \frac{\uvk^{1+\alpha}}{\uvk^\alpha+ \varepsilon_k^\alpha}(\pd_x^4 \uvk- \pd_x^2 \uvk ) \ud x \ud t\\
  =&\int_0^T\int_{\Dld} \phi \frac{\uvk^{1+\alpha}}{\uvk^\alpha+ \varepsilon_k^\alpha}(\pd_x^4 \uvk- \pd_x^2 \uvk ) \ud x \ud t \\
  & \hspace{1cm} +\int_0^T\int_{\Dgd} \phi \frac{\uvk^{1+\alpha}}{\uvk^\alpha+ \varepsilon_k^\alpha}(\pd_x^4 \uvk- \pd_x^2 \uvk ) \ud x \ud t\\
  =:& I_1+I_2.
\end{aligned}
\end{equation}
Then we estimate $I_1$ and $I_2$ separately.

For $I_1$, from \eqref{del1}, we have
\begin{equation}
  \vert \uvk(t,x) \vert\leq \frac{3\delta}{2}, \text{ for }t\in[0,T],\,x\in \Dld.
\end{equation}
Hence by H\"older's inequality, we know
\begin{align}\label{11I1}
 I_1\leq& \Big[\int_0^T\int_{\Dld} \Big(\phi \sqrt{\frac{\uvk^{1+\alpha}}{\uvk^\alpha+ \varepsilon_k^\alpha}} \Big)^2 \ud x \ud t \Big]^{\frac{1}{2}}\\
 &\cdot \Big[\int_0^T\int_{\Dld} \Big( \sqrt{\frac{\uvk^{1+\alpha}}{\uvk^\alpha+ \varepsilon_k^\alpha}}(\pd_x^4 \uvk- \pd_x^2 \uvk ) \Big)^2 \ud x \ud t \Big]^{\frac{1}{2}}\nn\\
 \leq& { C(\|u_0\|_{H^2})\|\phi\|_{L^\infty([0,T]\times\I)} \Big(\text{meas}\big\{(t,x);\,\vert\uvk\vert\leq\frac{3\delta}{2}\big\}\Big)^{\frac{1}{2}} }\nn\\
  \leq&  C(\|u_0\|_{H^2})T^{\frac{1}{2}}\frac{1}{(-\ln\delta)^{\frac12}}.\nn
\end{align}
Here we used \eqref{Ev01} and \eqref{supnorm} in the second inequality and \eqref{mes} in the last inequality.

Now we turn to estimate $I_2$.
Denote
\begin{equation}\label{Bd}
B_\delta:=\bigcup_{t\in[0,T]} \{t\}\times\Dgd.
\end{equation}
From \eqref{del1}, we know
\begin{equation}
  \uvk{ (t,x)}>\frac{\delta}{2}, {\text{  for }(t,x)\in B_{\delta}.}
\end{equation}
This, together with \eqref{n41}, shows that
\begin{equation}\label{temp711}
\begin{aligned}
  &\frac{\big(\frac{\delta}{2}\big)^{1+\alpha}}{\vark^\alpha+\big(\frac{\delta}{2}\big)^{\alpha}}\int\int_{B_\delta}(\pd_x^4 \uvk- \pd_x^2 \uvk )^2 \ud x \ud t \\
  \leq&  \int_0^T\int_\I {\frac{\uvk^{1+\alpha}}{\uvk^\alpha+ \varepsilon_k^\alpha}}(\pd_x^4 \uvk- \pd_x^2 \uvk )^2 \ud x \ud t\leq C(\|u_0\|_{\Hper^2(\I)}).
  \end{aligned}
\end{equation}
From \eqref{temp711}, there exists a subsequence of $\uvk$ (still denote as $\uvk$) and $w\in L^2(B_{\delta}) $ such that
\begin{equation}\label{tmbd}
\pd_x^4 \uvk- \pd_x^2 \uvk  \rightharpoonup w , \text{ in }L^2(B_{\delta}).
\end{equation}
Due to \eqref{strong1}, we know $w=\pd_x^4 u- \pd_x^2 u$.

On the other hand from \eqref{strong1}, we have
\begin{equation}
\left| \frac{\uvk^{1+\alpha}}{\uvk^\alpha+ \varepsilon_k^\alpha} - \frac{u^{1+\alpha}}{ u^\alpha+ \varepsilon_k^\alpha} \right| \leq (1+\alpha) |u-\uvk|\to 0.
\end{equation}
Then from $\frac{u^{1+\alpha}}{u^\alpha+ \varepsilon_k^\alpha} \to u$ as $\varepsilon_k \to 0$, we have
\begin{equation}
\left| \frac{\uvk^{1+\alpha}}{\uvk^\alpha+ \varepsilon_k^\alpha}-u \right|\leq \left| \frac{\uvk^{1+\alpha}}{\uvk^\alpha+ \varepsilon_k^\alpha} - \frac{u^{1+\alpha}}{ u^\alpha+ \varepsilon_k^\alpha} \right| + \left|\frac{u^{1+\alpha}}{u^\alpha+ \varepsilon_k^\alpha} -u \right| \to 0.
\end{equation}
This, together with \eqref{tmbd}, we have for $B_\delta$ defined in \eqref{Bd},
\begin{equation}\label{11I2}
  I_2=\int\int_{B_{\delta}}\phi {\frac{\uvk^{1+\alpha}}{\uvk^\alpha+ \varepsilon_k^\alpha}}(\pd_x^4 \uvk- \pd_x^2 \uvk )\ud x \ud t \rightarrow \int\int_{B_{\delta}}\phi u (\pd_x^4 u- \pd_x^2 u) \ud x \ud t.
\end{equation}
This shows there exists $K_2>K_1$ large enough such that for $k>K_2,$
\begin{equation}\label{n-I2}
\left| I_2 - \int\int_{B_{\delta}}\phi u (\pd_x^4 u- \pd_x^2 u) \ud x \ud t  \right|\leq \frac{1}{(-\ln\delta)^{\frac12}} .
\end{equation}
Recall $B_\delta$ defined in \eqref{Bd} and $I_1, I_2$ defined in \eqref{n64}. Combining \eqref{11I1} and \eqref{n-I2}, we know  for $k>K_2,$
\begin{equation}
\begin{aligned}
&\Big\vert \int_0^T\int_\I \phi {\frac{\uvk^{1+\alpha}}{\uvk^\alpha+ \varepsilon_k^\alpha}}(\pd_x^4 \uvk- \pd_x^2 \uvk )\ud x \ud t-\int\int_{B_\delta} \phi u (\pd_x^4 u- \pd_x^2 u) \ud x \ud t \Big\vert\\
= & \left| I_1 + I_2 - \int\int_{B_\delta} \phi u (\pd_x^4 u- \pd_x^2 u) \ud x \ud t   \right| \\
\leq& { \left[C(\|u_0\|_{H^2})T^{\frac{1}{2}}+ 1\right]\frac{1}{(-\ln\delta)^{\frac12}} },
\end{aligned}
\end{equation}
which implies that
\begin{align*}
& \lim_{\delta\rightarrow 0^+}\lim_{k\rightarrow\infty}\Big[ \int_0^T\int_\I \phi {\frac{\uvk^{1+\alpha}}{\uvk^\alpha+ \varepsilon_k^\alpha}}(\pd_x^4 \uvk- \pd_x^2 \uvk )\ud x \ud t \\
& \hspace{3cm} -\int\int_{B_\delta} \phi u (\pd_x^4 u- \pd_x^2 u) \ud x \ud t \Big]=0.
\end{align*}
{For any $\ell\geq 1$, assume the sequence $\delta_\ell\rightarrow 0$. Thus we can choose a sequence $\varepsilon_{\ell k}\rightarrow 0.$ Then by the diagonal argument, we have
$$\delta_\ell\rightarrow 0,\quad \varepsilon_{\ell\ell}\rightarrow 0,$$
as $\ell$ tends to $+\infty$.}
Notice $$P_T=\bigcup_{\delta>0}B_{\delta}.$$
We have
{  \begin{align*}
&\lim_{\ell\rightarrow\infty} \int_0^T\int_\I \phi \frac{u_{\varepsilon_{\ell\ell}}^{1+\alpha}}{u_{\varepsilon_{\ell\ell}}^\alpha+ \varepsilon_{\ell\ell}^\alpha}(\pd_x^4 u_{\varepsilon_{\ell\ell}}- \pd_x^2 u_{\varepsilon_{\ell\ell}}) \ud x\ud t  \\
=&\lim_{\ell\rightarrow\infty}\int\int_{B_{\delta_\ell}} \phi u (\pd_x^4 u- \pd_x^2 u) \ud x \ud t\\
=&\int\int_{P_T} \phi u (\pd_x^4 u- \pd_x^2 u) \ud x\ud t,
\end{align*}
}
which completes the proof.
\end{proof}

\subsubsection{Proof of energy dissipation laws in Theorem \ref{global_th} by taking the limit in the energy identities \eqref{Ev01} and in \eqref{Fv01}}
To take the limit in \eqref{Ev01}, we need a similar lemma whose proof is exactly same as Lemma \ref{claim3.5}.
\begin{lemma}\label{claim3.5-n}
  For $P_T$ defined in \eqref{PT}, { for any function $\phi\in C^\infty([0,T]\times\I),$} we have
  \begin{equation}\label{claim-n}
     \int_0^T\int_\I \phi \frac{\uv^{\frac12+\alpha}}{\uv^\alpha+ \varepsilon^\alpha}(\pd_x^4 \uv- \pd_x^2 \uv )\ud x \ud t \rightarrow  \int\int_{P_T} \phi u^{\frac12} (\pd_x^4 u- \pd_x^2 u) \ud x \ud t,
  \end{equation}
  as $\varepsilon\rightarrow 0.$
\end{lemma}
First recall the regularized solution $\uv$ satisfies the energy-dissipation equality \eqref{Ev01}, i.e.,
\begin{equation*}
      E(\uv(\cdot,T))+2\int_0^T\int_\I \frac{\uv^{1+\alpha}}{\uv^{\alpha+\varepsilon}}(\pd_x^4 \uv- \pd_x^2 \uv )^2\ud x \ud t =E(\uv(\cdot,0)).
    \end{equation*}
From Lemma \ref{claim3.5-n}, we have
$$\frac{\uv^{\frac12+\alpha}}{\uv^\alpha+ \varepsilon^\alpha}(\pd_x^4 \uv- \pd_x^2 \uv ) \rightharpoonup  u^{\frac12} (\pd_x^4 u- \pd_x^2 u), \text{ in }L^2(P_T).$$
Then by the lower semi-continuity of norm, we know
\begin{equation}\label{ss1}
\begin{aligned}
  \int\int_{P_T}u  (\pd_x^4 u- \pd_x^2 u)^2\ud x\ud t\leq& \liminf_{\varepsilon\rightarrow 0} \int\int_{P_T} \frac{\uv^{1+2\alpha}}{(\uv^\alpha+ \varepsilon^\alpha)^2}(\pd_x^4 \uv- \pd_x^2 \uv )^2 \ud x\ud t\\
  \leq&\liminf_{\varepsilon\rightarrow 0} \int\int_{P_T} \frac{\uv^{1+\alpha}}{\uv^\alpha+ \varepsilon^\alpha}(\pd_x^4 \uv- \pd_x^2 \uv )^2 \ud x \ud t.
  \end{aligned}
\end{equation}
Also from \eqref{es1} and lower semi-continuity of norm, we have
\begin{equation}\label{ss2}
  E(u(t,\cdot))\leq \liminf_{\varepsilon\rightarrow 0} E(\uv(t,\cdot)), \text{ for } t\in[0,T].
\end{equation}
Combining \eqref{Ev01}, \eqref{ss1} and \eqref{ss2}, we obtain
$$E(u(T,\cdot))+2\int\int_{P_T}u  (\pd_x^4 u- \pd_x^2 u)^2\ud x\ud t \leq E(u(0,\cdot)).$$

{
Second, recall the regularized solution $\uv$ satisfies the energy-dissipation equality \eqref{Fv01}, i.e.,
\begin{equation*}
      F_\varepsilon(\uv(T,\cdot))+\int_0^TE(\uv(t,\cdot)) \ud t =F_\varepsilon(\uv(0,\cdot)).
\end{equation*}
From \eqref{es1} and the lower semi-continuity of norm, we know
\begin{equation}\label{add18_2}
 \begin{aligned}
  &\int_0^TE(u(t,\cdot))\ud t\leq \liminf_{\varepsilon\rightarrow 0} \int_0^T E(\uv(t,\cdot))\ud t,\\
  &F(u(t,\cdot))\leq \liminf_{\varepsilon\rightarrow 0} F(\uv(t,\cdot)), \text{ for any }t\in[0,T].
  \end{aligned}
\end{equation}
Recall $F_\varepsilon(\uv)=\int_\I \frac{\varepsilon^\alpha}{1-\alpha} \frac{1}{u^{\alpha-1}}+ \uv\ud x$ and $0<\alpha<1$.
Then from the strong convergence \eqref{strong} we know 
$$\lim_{\varepsilon\to 0}F_\varepsilon(\uv)=F(u).$$
 This, together with \eqref{add18_2}, implies
$$F(u(T,\cdot))+\int_0^TE(u(t,\cdot)) \ud t \leq F(u(0,\cdot)).$$
Hence we complete the proof of Theorem \ref{global_th}.
}
\subsection{Long time behavior} We finally prove all the weak solution obtained in Theorem \ref{global_th} will converge to a constant as time goes to infinite. However, as explained in Remark \ref{rem2}, we can not characterize the limit constant uniquely.
\begin{theorem}\label{long2}
Under the same assumptions of Theorem \ref{global_th}, for every weak solution $u$ obtained in Theorem \ref{global_th}, there exists a constant $u^\star$ such that,
   as time $t\rightarrow +\infty,$ $u$ converges to $u^\star$ in the sense
\begin{equation}
 \|u(t, \cdot)-u^\star\|_{ H^2(\I)}\rightarrow 0, \text{  as }t\rightarrow +\infty.
\end{equation}
\end{theorem}
\begin{proof}
First, from the energy dissipation \eqref{F01}, we have for any $T>0$
\begin{equation}
TE(u(T))\leq F(u_0)-F(u(T)),
\end{equation}
which implies
\begin{equation}
E(u(t, \cdot))\leq \frac{1}{t}F(u_0) = \frac{c}{t}\to 0, \text{ as } t\rightarrow +\infty.
\end{equation}

Second, since \eqref{high} and \eqref{supnorm} are uniform in time, we actually have
\begin{equation}
\|u\|_{L^\8(0,+\8; \Hper^2(\I))}\leq c(u_0).
\end{equation}
Notice $H^2(\I)\hookrightarrow H^1(\I)$ compactly.
Then there exists a subsequence $t_{n}\to + \8$ and $u^\star \text{  in } H^1(\I)$ such that
\begin{equation}\label{802f}
  u(t_{n},{\cdot})\rightarrow u^\star({\cdot}), \text{  in } H^1(\I) \text{ as }t_{n}\rightarrow +\infty.
\end{equation}
On the other hand, since $E(u)$ is strictly convex in $\dot{H}^2$ we know $E$ has a unique critical point $w$ in $\dot{H}^2$ and
\begin{equation}
E(u(t, \cdot))\to \int_\I (\pd^2_x w)^2+(\pd_x w)^2 \ud x = 0 \text{ as }t \to +\8.
\end{equation}
Therefore $w=u^*=\text{const}$ and
\begin{equation}
  u(t,{\cdot})\rightarrow u^\star({\cdot}), \text{  in } H^2(\I) \text{ as }t\rightarrow +\infty.
\end{equation}
\end{proof}
\begin{remark}\label{rem2}
We mention that one cannot characterize the limit constant by the dissipation law \eqref{E01} since the dissipation term $\int\int_{P_T}2u(\pd_x^4 u - \pd_x^2 u)^2 \ud x \ud t$ holds only on $P_T$. Neither can we characterize the limit constant by conservation law in Observation 3 since it holds only for strict positive $u$ and we do not know how much we lose when $u$ touches zero.
\end{remark}

%First we give the definition of weak solution to regularized problem \eqref{equv}.
%\begin{defn}\label{defnuv}
%  For any fixed $\varepsilon>0,\,T>0$, we call a non-negative function $\uv(t,h)$ with regularities
%  \begin{equation}
%    \uv^3\in \Lty([0,T];{\blue \Hper^2}([0,1])),\quad \frac{\uv^3}{\sqrt{\varepsilon+\uv^2}}(\uv^3)_{hhhh}\in L^2(0,T;{\blue \Lper^2}([0,1])),
%  \end{equation}
%  \begin{equation}
%    \uvt\in L^2([0,T];{\blue \Lper^2}([0,1])),\quad \uv^3\in C([0,T];\Hper^1([0,1])),
%  \end{equation}
%   weak solution to regularized problem \eqref{equv} if
%  \begin{enumerate}[(i)]
%    \item for any function $\phi\in C^\infty([0,T]\times[0,1])$, $\uv$ satisfies
%    \begin{equation}\label{eqv01}
%      \int_0^T\int_\I \phi \uvt \ud x  \ud t+\int_0^T\int_\I \phi\frac{\uv^4}{\varepsilon+\uv^2}(\uv^3)_{hhhh}\ud x  \ud t =0;
%    \end{equation}
%    \item the following {\blue first} energy-dissipation equality holds
%    \begin{equation}\label{Ev01}
%      E(\uv(T,\cdot))+\int_0^T\int_\I\Big[\frac{\uv^3}{\sqrt{\varepsilon+\uv^2}}(\uv^3)_{hhhh}\Big]^2 \ud x  \ud t =E(\uv(0,\cdot)).
%    \end{equation}
%    {\blue \item the following second energy-dissipation equality holds
%    \begin{equation}\label{Fv01}
%      F_\varepsilon(\uv(T,\cdot))+6\int_0^TE(\uv(t,\cdot)) \ud t =F_\varepsilon(\uv(0,\cdot)),
%    \end{equation}
%    where
%    $F_\varepsilon(\uv):=\int_\I \varepsilon \ln |\uv| \ud x +F(\uv)$
%    is a perturbed version of $F$.  }
%  \end{enumerate}
%\end{defn}

\appendix

\section{Small data existence in the Wiener algebra}
\label{sec:aprori}
This section will follow the Weiner algebra framework established in \cite{GBM2018,ambrose} in periodic settings and \cite{LS2018} on $\mathbb{R}^{d}$, $d \geq 1$, for the fully 4th order model. Since the general framework is very similar, we just state the main results in Section \ref{sec:mainResults} and give the key estimates.  { Since these results are easy to state and prove in general dimension, we will present them as such.}

\subsection{Notation}
We introduce the following useful norms:
\bea\label{weightednorm}
\|f\|_{\mathcal{\dot{F}}^{s,p}}^{p}(t) \eqdef \int_{\mathbb{R}^d} |\xi|^{sp} |\hat{f}(\xi,t)|^{p} d\xi, \quad s > -d/p,\quad 1 \le p\le 2.
\eea
We note that the Wiener algebra $A(\mathbb{R}^d)$ is $\mathcal{\dot{F}}^{0,1}$, and the condition $\Delta h_0 \in A(\mathbb{R}^d)$ is given by $h_0 \in \mathcal{\dot{F}}^{2,1}$.  Here $\hat{f}$ is the standard Fourier transform of $f$:
\begin{equation} \label{fourierTransform}
\hat{f}(\xi) \eqdef  \mathcal{F}[f](\xi) = \frac{1}{(2\pi)^{d/2}}\int _{\mathbb{R}^{d}} f(x) e^{- i x\cdot \xi} dx.
\end{equation}
When $p=1$ we denote the norm by
\begin{equation} \label{Snorm}
\|f\|_{s} \eqdef \int_{\mathbb{R}^{d}} |\xi|^{s}|\hat{f}(\xi)| \ d\xi.
\end{equation}
We will use this norm generally for $s>-d$ and we refer to it as the \textit{s-norm}.
Notice that
for any $n\in\mathbb{N}$ we have
\ba
|\D^n f(x)|\leq  \int_{\mathbb{R}^d} |\xi|^{n} |\hat{f}(\xi)| d\xi = \|f\|_n.
\ea

To further study the case $s = -d$, then for $s\ge -d$ we define the  \textit{Besov-type s-norm}:
\begin{equation} \label{normSinfty}
\|f\|_{s,\infty} \eqdef \Big\|\int_{C_{k}} |\xi|^{s}|\hat{f}(\xi)| \ d\xi\Big\|_{\ell^{\infty}_{k}}
=
\sup_{k \in \mathbb{Z}} \int_{C_{k}} |\xi|^{s}|\hat{f}(\xi)| \ d\xi,
\end{equation}
where for $k \in \mathbb{Z}$ we have
\begin{equation} \label{Cj}
C_{k} = \{\xi \in \mathbb{R}^d : 2^{k-1}\leq |\xi| < 2^{k}\} \,.
\end{equation}
Note that we have the inequality
\begin{equation}\label{ineqbd}
\|f\|_{s,\infty}  \leq \int_{\mathbb{R}^{d}} |\xi|^{s}|\hat{f}(\xi)| \ d\xi = \|f\|_{s}.	
\end{equation}
We note that $$\|f\|_{-d/p,\infty} \lesssim \|f\|_{L^p(\mathbb{R}^d)}$$ for $p\in [1,2]$ as is shown in \cite[Lemma 5]{MR3683311}.

Further, when $p=2$ we denote the norm (for $s > -d/2$) by
\bea\label{weighted2norm}
\|f\|_{\mathcal{\dot{F}}^{s,2}}^{2} \eqdef \int_{\mathbb{R}^d} |\xi|^{2s} |\hat{f}(\xi)|^{2} d\xi
=   \|f\|_{\dot{H}^s}^2
    = \|(-\Delta)^{s/2}f\|_{L^2(\mathbb R^d)}^2.
\eea
%We also introduce following norms with analytic weights:
%\bea\label{analyticnorm}
%\|f\|_{\mathcal{\dot{F}}^{s,p}_{\nu}}^{p}(t) \eqdef \int_{\mathbb{R}^d} |\xi|^{sp}e^{p\nu(t) |\xi|}|\hat{f}(\xi,t)|^{p} d\xi,
%\quad s\geq 0,\quad  p\in  [1,2],
%\eea
%for a positive function $\nu(t)$.   %This notation was used in \cite{GGJPS}.

We also introduce the following notation for an iterated convolution
$$
f^{*2}(x) = (f * f)(x) = \int_{\mathbb{R}^d} f(x-y) f(y) dy,
$$
where $*$ denotes the standard convolution in $\mathbb{R}^d$.  Furthermore in general
$$
f^{*j}(x) = (f * \cdots * f)(x),
$$
where the above contains $j-1$ convolutions of $j$ copies of $f$.  Then by convention when $j=1$ we have
$f^{*1} =f$, and further we use the convention $f^{*0} =1$.

We additionally use the notation $A \lesssim B$ to mean that there exists a positive inessential constant $C>0$
such that $A \le C B$.    The notation $\approx$ used as $A \approx B$ means that both $A \lesssim B$  and $B \lesssim A$  hold.

\subsection{Main results}  \label{sec:mainResults}

We have the following results for small, highly regular data.

\begin{theorem}\label{Main.Theorem}
Consider initial data $h_0 \in \mathcal{\dot{F}}^{0,2}\cap \mathcal{\dot{F}}^{2,1}$ further satisfying 
\[ \| h_0 \|_{2} < y_* \]
  where $y_*>0$ is given explicitly in Remark \ref{constant.size}.  Then there exists a global in time unique solution to \eqref{pde} given by $h(t) \in C^0_t (\mathcal{\dot{F}}^{0,2}\cap\mathcal{\dot{F}}^{2,1})$ and we have that
\begin{equation}\label{21normbound}
 \| h \|_2(t)  + \sigma_{2,1} \int_0^t   \|  h \|_6(\tau) d\tau
  \le    \| h_0 \|_2
\end{equation}
with $\sigma_{2,1}>0$ defined by \eqref{sigma.constant.def}.
\end{theorem}

In the next remark we explain the size of the constant.

\begin{remark}\label{constant.size}  We can compute precisely the size of the constant $y_*$ from Theorem \ref{Main.Theorem}.  In particular the condition that it should satisfy is that
$$
f_2(y_*)   = (y_*^3 + 6y_*^2+7y_*+1)e^{y_*}-1
=
\sum_{j=1}^{\infty}  \frac{(j+1)^3  }{j!}  y_*^{j}<1.
$$
This is identical to the threshold found in \cite{LS2018}, the reason for which will be born out below.  Note, this constant can likely be sharpened as in the work of \cite{ambrose}.
\end{remark}

Now in the next theorem we prove the large time decay rates, and the propagation of additional regularity, for the solutions above.

\begin{theorem}\label{Cor.Main.Theorem}
Given the solution to \eqref{pde} from Theorem \ref{Main.Theorem}.  Suppose additionally that $\| h_0 \|_{s_1} <\infty$ and $\| h_0 \|_{s_2} <\infty$ for some $-1<s_1< s_2$. Then we have the following uniform decay estimate for $t\ge 0$:
\begin{equation}\label{fastest.decay.rate}
\| h \|_{s_2} \lesssim  (1+t)^{-\frac{s_2-s_1}{2}}.
\end{equation}
The implicit constant in the inequality above depends on $\| h_0 \|_2$, $\| h_0 \|_s$.
\end{theorem}

We will only prove Theorems \ref{Main.Theorem} and \ref{Cor.Main.Theorem} in the following sections.  The proof will follow very closely the framework from \cite{LS2018,ambrose} so we just show the key estimates.

\subsection{Proof of Theorem \ref{Main.Theorem}}
In this section we prove the apriori estimates for the exponential PDE in \eqref{pde} and \eqref{pdeSum} in the spaces $\mathcal{\dot{F}}^{s,1}$.  The key point to the global in time classical solution is that we can prove a global in time Lyapunov inequality  \eqref{energy3} below under an $O(1)$ medium size smallness condition on the initial data.
\subsubsection{A priori estimate in $\mathcal{\dot{F}}^{2,1}$}
  We first establish the case of $\mathcal{\dot{F}}^{2,1}$ in order to explain the main idea in the simplest way.   The equation \eqref{pde} can be recast by Taylor expanion as
\begin{equation}\label{pdeSum}
   h_t + \Delta^2 h -\Delta h = \Delta \sum_{j=2}^{\infty}  \frac{(-\Delta h)^j}{j!}-\sum_{j=2}^{\infty}  \frac{(-\Delta h)^j}{j!}.
\end{equation}
We look at this equation \eqref{pdeSum} using the Fourier transform \eqref{fourierTransform} so that equation (\ref{pde}) is expressed as
\begin{align} \label{Fourier1}
   \partial_t  \hat h(\xi,t) & +  |\xi|^4 \hat h(\xi,t) +  |\xi|^2 \hat h(\xi,t) \\
   & = -    |\xi|^2\sum_{j=2}^{\infty}  \frac{1}{j!}  (|\cdot|^2 \hat h)^{*j} (\xi,t)  -\sum_{j=2}^{\infty}  \frac{1}{j!}  (|\cdot|^2 \hat h)^{*j} (\xi,t) \,. \notag
\end{align}
We multiply the above by $|\xi|^2$ to obtain
\begin{align} \label{Fourier2}
   \partial_t  |\xi|^2 \hat h(\xi,t) & + |\xi|^6 \hat h(\xi,t) + |\xi|^4 \hat h(\xi,t)
   = \\
  & -    |\xi|^4 \sum_{j=2}^{\infty}  \frac{1}{j!}  (|\cdot|^2 \hat h)^{*j} (\xi,t)
   -    |\xi|^2 \sum_{j=2}^{\infty}  \frac{1}{j!}  (|\cdot|^2 \hat h)^{*j} (\xi,t) . \notag
\end{align}
We will estimate this equation on the Fourier side in the following.

Our first step will be to estimate the infinite sum in \eqref{Fourier2}.   To this end notice that for any real number $s \ge 0$ the following triangle inequality holds:
\begin{equation}
\label{triangleS}
    |\xi|^{s}
    \le   j^{(s-1)^+} ( |\xi-\xi_1|^{s} + \cdots + |\xi_{j-2}-\xi_{j-1}|^{s} + |\xi_{j-1}|^{s}),
\end{equation}
where
$(s-1)^+ = s-1$ if $s\ge 1$ and  $(s-1)^+ = 0$ if $0 \le s \le 1$.
%$$
%\cSdef
%\eqdef
%\begin{cases}
%1 & \text{if}  ~0 \le s \le 1,  \\
%  j^{s-1} &  \mbox{when} ~   s \ge 1.
%\end{cases}
%$$
We have further using the inequality \eqref{triangleS} when $s\ge 1$ that
\begin{align} \label{convolution1}
 \int_{\mathbb{R}^d} |\xi|^{s} | (|\cdot|^2 \hat h)^{*j} (\xi)|  \,d\xi
&  \le  j^{s}  \int_{\mathbb{R}^d} |(|\cdot|^{s+2} \hat h)*(|\cdot|^2 \hat h)^{*(j-1)} |  \,d\xi
 \\
 & \le j^{s} \| h \|_{s+2} \|  h \|_2^{j-1}. \notag
\end{align}
Above we used Young's inequality repeatedly with $1+1 = 1+1$.

Observe that generally $\partial_t |{\hat h}| = \frac{1}{2}\Big(\partial_t \hat{h}\overline{\hat{h}}+\hat{h}\overline{\partial_t\hat{h}}\Big)|{\hat h}|^{-1}$.  Now we multiply  \eqref{Fourier2} by $\overline{\hat h} |{\hat h}|^{-1}(\xi,t)$, add the complex conjugate of the result, then integrate, and use (\ref{convolution1}) for $s=4$ and $s=2$ to obtain the following differential inequality
\begin{equation} \label{energy1}
  \frac{d}{dt} \| h \|_2 + \|  h \|_6 + \|  h \|_4
  \le \| h \|_6 \sum_{j=2}^{\infty}  \frac{j^4}{j!} \|  h \|_2^{j-1} + \| h \|_4 \sum_{j=2}^{\infty}  \frac{j^2}{j!} \|  h \|_2^{j-1} .
\end{equation}
Now we denote the function
\begin{equation}\label{infinite.series.fcn}
   f_2(y) := \sum_{j=2}^{\infty}  \frac{j^{4}  }{j!} y^{j-1} = \sum_{j=1}^{\infty}  \frac{(j+1)^3  }{j!}  y^{j}
\end{equation}
and
\begin{equation}
   f_0(y) := \sum_{j=2}^{\infty}  \frac{j^{2}  }{j!} y^{j-1} = \sum_{j=1}^{\infty}  \frac{(j+1)  }{j!}  y^{j}\leq f_2(y).
\end{equation}
Then \eqref{infinite.series.fcn} defines an entire function which is strictly increasing for $y \ge 0$ with $f_2(0)=0$.  In particular we choose the value $y_*$ such that
$f_2(y_*)=1$.

Then (\ref{energy1}) can be recast as
\begin{equation} \label{energy2}
  \frac{d}{dt} \| h \|_2 + (\| h \|_6+ \|h\|_4)
  \le  (\| h \|_6+ \|h\|_4) f_2 \big( \| h \|_2 \big).
\end{equation}
If the initial data satisfies
\begin{equation} \label{condition1}
 \| h_{0} \|_2 < y_*,
\end{equation}
then we can show that $ \|  h  \|_2(t)$ is a decreasing function of $t$.
{ Note that $y_* = y_{2*}$ in the notation from \eqref{condition2} below. }  In particular
$$
f_2 \big( \|  h  \|_2(t) \big)
\le
f_2 \big( \| h_{0} \|_2 \big) < 1.
$$
Using this calculation then (\ref{energy2}) becomes
\begin{equation} \label{energy3}
  \frac{d}{dt} \| h \|_2
  +
  \sigma_{2,1} \|  h \|_6
  \le   0,
\end{equation}
where
\begin{equation}\label{sigma.constant.def}
\sigma_{2,1} \eqdef 1-f_2( \|  h_{0} \|_2 ) >0.
\end{equation}
In particular if \eqref{condition1} holds, then $\|  h  \|_2(t) < y_*$ will continue to hold for a short time, which allows us to establish \eqref{energy3}.    The inequality (\ref{energy3}) then defines a free energy and shows the dissipation production.

At the end of this section we look closer at the function $f_2(y)$:
$$
   f_2(y) = \sum_{j=1}^{\infty}  \frac{(j+1)^3  y^{j}}{j!}
   = \sum_{j=1}^{\infty}  \frac{( j(j-1)(j-2) + 6j(j-1) + 7j +1)  y^{j}}{j!},
$$
which gives
\begin{equation}\label{fn1}
f_2(y)   = (y^3 + 6y^2+7y+1)e^y-1 \,.
\end{equation}
We know that $f_2(0)=0$ and $f_2(y)$ is strictly increasing. Let $y_*$ satisfy
\begin{equation}  \label{fn2}
 (y_{*}^3+ 6y_{*}^2+7y_*+1)e^{y_*}-1 = 1.
\end{equation}
Then $f_2(y_*) = 1$ as above.

\subsubsection{A priori estimate in the high order s-norm}\label{sec:highorder}
In this section we prove a high order estimate for any real number $s> -1$. 

To extend this analysis to the case where $s \ne 2$ we consider infinite series:
\begin{equation}
\label{infinite.series.fcn.s}
   f_s(y) = \sum_{j=2}^{\infty}  \frac{j^{s+2} }{j!}  y^{j-1} = \sum_{j=1}^{\infty}  \frac{(j+1)^{s+1}  }{j!} y^{j}.
\end{equation}
Again $f_s(0)=0$ and $f_s(y)$ is a strictly increasing entire function for any real $s$. We further remark that for $r \ge s$ we have the inequality
\begin{equation}
\label{infinite.series.fcn.y.s}
   f_s(y) \le f_r(y), \quad s \le r, \quad \forall y\ge 0.
\end{equation}
We further have a simple recursive relation
$$
   f_s(y) = \frac{d}{dy} \big(y f_{s-1}(y) \big) , \quad
   f_{-1}(y) =  e^y - 1.
$$
This allows us to compute $f_s(y)$ for any $s$ a non-negative integer as in \eqref{fn1}.

Similar  to the previous section, we have
 \begin{align} \label{Fourier3}
   \partial_t  |\xi|^{s} \hat h(\xi,t) & + |\xi|^{s+4} \hat h(\xi,t) + |\xi|^{s+2} \hat h(\xi,t) \\
 &   = -  |\xi|^{s+2} \sum_{j=2}^{\infty}  \frac{1}{j!} (|\xi|^2 \hat h)^{*j} (\xi,t)
   -  |\xi|^{s} \sum_{j=2}^{\infty}  \frac{1}{j!} (|\xi|^2 \hat h)^{*j} (\xi,t) . \notag
\end{align}
Using (\ref{convolution1}) and (\ref{Fourier3}), one has
\begin{align} \label{energy4}
  \frac{d}{dt} \| h \|_s & + \|  h \|_{s+4}+ \|  h \|_{s+2}
  \le  \\
  & \|  h \|_{s+4} \sum_{j=2}^{\infty}  \frac{j^{s+2}}{j!} \| h \|_2^{j-1} + \|  h \|_{s+2} \sum_{j=2}^{\infty}  \frac{j^{s}}{j!} \| h \|_2^{j-1}.  \notag
\end{align}
Since $f_{s-1}(y)\leq f_s$ for any $y>0$, from definition of $f_s$ in \eqref{infinite.series.fcn.s}, we recast (\ref{energy4}) as
\begin{equation} \notag
  \frac{d}{dt} \|  h \|_s + \| h \|_{s+4} + \| h \|_{s+2}
  \le  (\| h \|_{s+4} + \| h \|_{s+2}) f_s \big( \| h \|_2\big).
\end{equation}
Let $y_{s*}$ satisfy $f_s(y_{s*}) = 1$.
If
\begin{equation} \label{condition2}
 \| h_{0} \|_2 < \min(y_{s*}, y_*).
\end{equation}
Note that by \eqref{infinite.series.fcn.y.s} we have that $y_{s*} \le y_{r*}$ for $s \le r$.  In particular we are using $y_{2*} = y_*$ in \eqref{condition1} and therefore $y_{s*} \le y_*$ whenver $s \le 2$.

Then by \eqref{energy3} we have
$$
f_s \big( \| h(\cdot,t) \|_2 \big)
\le
f_s \big( \| h_{0} \|_2 \big) < 1 \,.
$$
Hence we conclude the energy-dissipation relation
\begin{equation} \label{energy-diss2}
  \frac{d}{dt} \| h(\cdot,t) \|_s
   +
 \sigma_{s,1} (\| h(\cdot,t) \|_{s+4} + \| h(\cdot,t) \|_{s+2})\,
  \le  0 , \,
\end{equation}
when (\ref{condition2}) holds.  Here we define $\sigma_{s,1} \eqdef \big( 1-f_s( \|  h_{0} \|_2 \big)>0$.

 \subsection{Proof of the Theorem \ref{Cor.Main.Theorem}}\label{sec:largedecay}

In this section we prove the Theorem \ref{Cor.Main.Theorem} by the following decay lemma from Patel-Strain \cite{MR3683311}:

\begin{lemma}\label{decaylemma}
Suppose $g=g(t,x)$ is a smooth function with $g(0,x) = g_0(x)$ and assume that for some $\ss \in \mathbb{R}$, $\|g_0\|_{\ss} < \infty$ and
$$
\|g(t)\|_{\rr,\infty} \leq C_{0}
$$
for some $\rr \ge -\Ddim$ satisfying  $ \rr < \ss$.
Let the following differential inequality hold for $\gamma>0$ and for some $C>0$:
	$$
	\frac{d}{dt}\|g\|_{\ss} \leq -C \|g\|_{\ss +\gamma}.
	$$
	  Then we have the uniform in time estimate
	$$
	\|g\|_{\ss}(t) \lesssim \left( \|g_0\|_{\ss}+ C_{0}\right) (1+t)^{-(\ss-\rr)/\gamma}.
	$$
\end{lemma}

\begin{proof}[Proof of Theorem \ref{Cor.Main.Theorem}]
For any $s>-1$ from \eqref{energy-diss2} we have
\begin{equation} \label{energy-diss00}
  \frac{d}{dt} \| h(\cdot,t) \|_s
   \leq -
 \sigma_{s,1} ( \| h(\cdot,t) \|_{s+4}+ \| h(\cdot,t) \|_{s+2}) \,
   \,.
\end{equation}
From the assumption, $\| h_0 \|_{s_2} <\infty$ and $\| h_0 \|_{s_1} <\infty$ for some $s_2 >s_1>-1$. Therefore by
we know
\begin{equation}
\|h\|_{s_1,\infty}\leq \|h\|_{s_1}\leq \| h_0 \|_{s_1} <\infty.
\end{equation}
Then we can apply Lemma \ref{decaylemma} with $\mu=s_2, \gamma=2, \rho=s_1$ to see that \eqref{energy-diss00} implies
\begin{equation}
\|h\|_{s_2}\leq (\|h_0\|_{s_2}+C_0)(1+t)^{-\frac{s_2-s_1}{2}}.
\end{equation}
\end{proof}

\end{document}